\documentclass[a4paper,10pt]{amsart}
\usepackage[english]{babel}
\usepackage{amsmath, amssymb, amsthm, amscd}
\usepackage{enumerate}
\usepackage{palatino}
\usepackage{mathpazo}
\usepackage{paralist}
\usepackage[a4paper]{geometry}
\usepackage{url}
\usepackage{hyperref}
\usepackage{tikz}
\usepackage{tikz-cd}
\usepackage[all]{xy}
\usepackage{xcolor}
\usepackage{graphicx}
\usepackage{lscape}
\usepackage{mathrsfs}  

\DeclareMathOperator{\id}{id}

\DeclareMathOperator{\Crit}{Crit }

\DeclareMathOperator{\cl}{cl}

\newcommand{\Fix}{\mathrm{Fix\;}}

\renewcommand{\setminus}{\smallsetminus}

\newcommand{\NN}{\mathbb{N}}

\newcommand{\RR}{\mathbb{R}}
\newcommand{\CC}{\mathbb{C}}

\newcommand{\HH}{\mathbb{H}}

\renewcommand{\S}{\mathbb{S}}

\newcommand{\C}{\mathcal{C}}

\newcommand{\sing}{\mathrm{sing}}

\newcommand{\Ver}{\mathrm{Ver}}

\DeclareMathOperator{\TC}{\mathsf{TC}}

\DeclareMathOperator{\cat}{\mathsf{cat}}

\newcommand{\pr}{\mathrm{pr}}

\newcommand{\ver}{\mathrm{ver}}

\newcommand{\op}{\mathrm{op}}

\theoremstyle{plain}
\newtheorem{theorem}{Theorem}[section]
\newtheorem*{theorem*}{Theorem}
\newtheorem{prop}[theorem]{Proposition}
\newtheorem{lemma}[theorem]{Lemma}
\newtheorem{cor}[theorem]{Corollary}

\theoremstyle{definition}
\newtheorem{definition}[theorem]{Definition}

\theoremstyle{remark}
\newtheorem{remark}[theorem]{Remark}
\newtheorem{example}[theorem]{Example}

\numberwithin{equation}{section}

\begin{document}
\setlength{\parindent}{0cm}

\title{Approaches to critical point theory via sequential and parametrized topological complexity}


\author{Stephan Mescher}
\address{Institut f\"ur Mathematik \\ Martin-Luther-Universit\"at Halle-Wittenberg \\ Theodor-Lieser-Stra\ss{}e 5 \\ 06120 Halle (Saale) \\ Germany}
\email{stephan.mescher@mathematik.uni-halle.de}

\author{Maximilian Stegemeyer}
\address{Mathematisches Institut \\ Universit\"at Freiburg \\ Ernst-Zermelo-Stra\ss{}e 1 \\ 79104 Freiburg \\ Germany} 
\email{maximilian.stegemeyer@math.uni-freiburg.de}
\date{\today}

\begin{abstract}
	The Lusternik-Schnirelmann category of a space was introduced to obtain a lower bound of the number of critical points of a $C^1$-function on a given manifold. Related to Lusternik-Schnirelmann category and motivated by topological robotics, the topological complexity (TC) of a space is a numerical homotopy invariant whose topological properties are an active field of research. The notions of sequential and parametrized topological complexity extend the ideas of topological complexity. While the definition of TC is closely related to Lusternik-Schnirelmann category, the connections of sequential and parametrized TC to critical point theory have not been fully explored yet. In this article we apply methods from Lusternik-Schnirelmann theory to establish several lower bounds of numbers of critical points of functions in terms of sequential and parametrized TCs. We carry out several consequences and applications of these bounds, among them a computation of the parametrized TC of the unit tangent bundles of $(4m-1)$-spheres.
\end{abstract}

\maketitle

\section*{Introduction}
Relating topological invariants of spaces to numbers of critical points of functions is one of the central ideas in topology and geometry.
In the 1920s, Lusternik and Schnirelmann introduced their notion of \textit{category of a space} to find a lower bound of these numbers for arbitrary $C^1$-functions on a closed smooth manifold.
Since then, there have been many extensions and applications of Lusternik-Schnirelmann category to critical point theory, e.g. to infinite-dimensional Banach manifolds by R. Palais in \cite{PalaisLuster} and to equivariant settings by T. Bartsch in \cite{Bartsch} as well as M. Clapp and D. Puppe in \cite{ClappPuppeSymm}.

Another common approach to relate topological properties of manifolds to critical points of functions is Morse theory, see \cite{MilMorse}.
While Morse theory often produces stronger bounds of the numbers of critical points, its downside is that the considered function needs to be of class $C^2$ and its critical points need to fulfill a non-degeneracy condition, which is only satisfied generically. \medskip 

In recent years, invariants similar to Lusternik-Schnirelmann (LS) category have been studied in topological robotics. Motivated by the motion planning problem from robotics, M. Farber introduced the notion of \emph{topological complexity} of a topological space $X$ in \cite{farber:2003} and developed it further in several subsequent articles.
For a path-connected topological space $X$, let $PX = C^0([0,1],X)$ be the path space endowed with the compact-open topology and let 
$$\pi\colon PX\to X\times X, \quad \pi(\gamma) = (\gamma(0),\gamma(1)),$$ be the path fibration.
The topological complexity  of $X$ is defined as the smallest integer $k\in \mathbb{N}$ such that there exists an open cover $U_1,U_2,\ldots , U_k$ of $X\times X$ with $U_i$ admitting a continuous section of $\pi$ for each $i\in\{1,2,\ldots,k\}$. Here, we note that we use Farber's original definition instead of the reduced definition which is by one smaller than ours and which has become very common in the literature.

Since M. Farber's seminal work, many results on and many variants of topological complexity have been studied. In this article, we shall consider the \textit{sequential} or \textit{higher} topological complexity of a space $X$, denoted by $\TC_r(X)$ and the \textit{(sequential) parametrized topological complexity} of a fibration $p\colon E\to B$, denoted by $\TC_r[p\colon E \to B]$.

We want to briefly discuss these two concepts. Let $r \in \NN$ with $r\geq 2$ and let $X$ be a path-connected space. Then
$$  \pi_r\colon PX \to X^r, \quad    \pi_r (\gamma) = \big(      \gamma(0), \gamma(\tfrac{1}{r-1}),\ldots, \gamma(\tfrac{r-2}{r-1}),\gamma(1)   \big) , $$
is a fibration. The \textit{$r$-th sequential topological complexity of $X$} is defined as the smallest integer $k\in \mathbb{N}$ such that there exists an open cover $U_1,U_2,\ldots , U_k$ of $X^r$ such that each $U_i$, $i\in\{1,2,\ldots,k\}$ admits a continuous section of $\pi_r$. This notion was defined by Y. Rudyak in \cite{RudyakHigher}. Note that for $r=2$, we obtain that $\TC_2(X)=\TC(X)$, the topological complexity of $X$.

The notion of parametrized topological complexity was introduced by D. Cohen, M. Farber and S. Weinberger in \cite{cohen2021topology} and generalized to sequential parametrized topological complexities by M. Farber and A. Paul in \cite{FarberPaulI}. It has since been further studied by the same authors in \cite{FarberPaulII} and by M. Farber and J. Oprea in \cite{FarberOpreaSeqParam}. 

Given $r \in \NN$ with $r \geq 2$ and a fibration $p\colon E \to X$ with path-connected fibers, we consider
$$     E_X^r =  \{  (u_1,u_2,\ldots, u_r)\in E^r \mid p(u_1 ) =p(u_2)= \ldots = p(u_r) \}         $$
and$$  E_X^I  =     \{   \gamma\in PE \mid  p(\gamma(t)) = p(\gamma(0)) \,\,\,\text{for all}\,\,\, t\in [0,1]\} .       $$
As a consequence of the main result of \cite[Appendix]{CFWplane}, the map $$   \Pi_r\colon E_X^I \to E_X^r, \qquad   \Pi_r (\gamma) = \big(      \gamma(0), \gamma(\tfrac{1}{r-1}),\ldots, \gamma(\tfrac{r-2}{r-1}),\gamma(1)   \big)  , $$
is a fibration.
The $r$\textit{-th sequential parametrized topological complexity} of  $p\colon E\to X$, denoted by $\TC_r[p\colon E\to X]$ is then defined as the smallest value of $k\in \mathbb{N}$ such that there exists an open cover $U_1,U_2,\ldots , U_k$ of $E_X^r$ with each $U_i$, $i\in\{1,2,\ldots,k\}$ admitting a continuous section of $\Pi_{r}$.\medskip

The notions of topological complexity (TC) and its variants are strongly related to LS category, in fact all of these concepts are special cases of the notion of \emph{sectional category of a fibration}, which was introduced by A. Schwarz in \cite{SchwarzGenus} under the name of \emph{genus}. However, in contrast to LS category, the relations between TC and its variants on the one hand and critical point theory on the other hand have only been studied in a few instances.
In \cite[Definition 4.30]{FarberBook}, M. Farber defined the notion of a navigation function on a closed manifold $M$  to establish upper bounds of $\TC(M)$ in terms of critical points of functions on $M\times M$. There, a function $F\colon M\times M \to \RR$ is called a \textit{navigation function} if it is non-negative, if $F^{-1}(\{ 0\}) = \Delta_M$, where $\Delta_M\subset M\times M$ is the diagonal, and if it is a Morse-Bott function.
By \cite[Theorem 4.32]{FarberBook}, it holds for such a navigation function that 
$$     \TC(M)  \leq \sum_{i=1}^k \max\{ \TC_M(C)\ | \ C\subset M\times M\,\,\,\text{is a critical manifold with}\,\,\, F(C) = c_i\} + 1  ,  $$
where $c_1,\dots,c_k$ are the non-zero critical values of $F$. Here, $\TC_M(C)$ is the \textit{subspace complexity of the set} $C\subset M\times M$ as defined in \cite[Definition 4.20]{FarberBook}, a term that we shall discuss below in greater generality. Navigation functions on lens spaces were investigated by A. Costa in \cite{Costa}.\medskip

In this article we generalize the aforementioned approach to TC via critical point theory in several ways.
\begin{itemize}
	\item We follow the topological approach to LS theory of Y. Rudyak and F. Schlenk from \cite{RudyakSchlenk} to establish results on numbers of fixed points of self-maps homotopic to the identity. 
	\item We study upper bounds of sequential and parametrized TCs of sublevel sets in terms of critical points for arbitrary $C^1$-functions and obtain results for generalized navigation functions as special cases.
	\item In contrast to Farber's navigation functions, we do not make any assumptions of non-degeneracy on the critical points of the functions.
\end{itemize}

We want to point out that the interplay between critical points of Morse-Bott functions and LS-type invariants has been another fruitful line of research in recent years. In \cite{KadzisaMimura}, H. Kadzisa and M. Mimura give a refined upper bound of LS category in terms of cone-decompositions induced by differentiable functions. Applying their approach to certain Morse-Bott functions, they obtain new computations of LS categories of homogenous manifolds, e.g. for complex Stiefel manifolds. Another relevant result studies the relations between critical points of function and the \emph{homotopic distance} of maps, which was introduced by E. Mac\'{i}as-Virg\'{o}s and D. Mosquera-Lois in \cite{MVMLhomdist}. It generalizes various numerical invariants of spaces, among them TC and LS category. In \cite[Theorem 4.5]{MVMLgeneralized}, the two aforementioned authors and M. J. Pereira S\'{a}ez have established an upper bound of the homotopic distance between two maps on a manifold $M$, which, given a Morse-Bott function $\Phi\colon M \to \RR$, estimates the homotopic distance by the sum of the homotopic distances of the restrictions of the maps to the critical submanifolds of $\Phi$.   \medskip

Throughout the following, given a function $F\colon X \to \RR$, we let $F^\lambda:= F^{-1}((-\infty,\lambda])$ denote the corresponding closed sublevel set for each $\lambda \in \RR$.
The following theorem is our most general result for sequential topological complexities. All of the terminology used in these assertions will be made precise below. For the moment, we just note that condition (D) mentioned in these statements is a topological generalization of the Palais-Smale condition and that the sequential analogues of subspace complexities occurring in the following theorem have first been studied by N. Daundkar in \cite{DaundkarGroup}.

\begin{theorem*}[Theorem \ref{TheoremTCLS}]
    Let $r \in \NN$ with $r \geq 2$ and let $X$ be a path-connected ANR. Let $\varphi\colon X^r \to X^r$ be continuous and homotopic to the identity, let $F\colon X^r \to \RR$ be a Lyapunov function for $\varphi$ which is bounded from below. Assume that $F(\Fix \varphi)$ is discrete and that $(\varphi,F)$ satisfies condition (D).
\begin{enumerate}[a)]
\item For all $\lambda \in \RR$, it holds that 
	$$\TC_{r,X}(F^\lambda) \leq \sum_{\mu \in (-\infty,\lambda]} \TC_{r,X}(F_\mu \cap \Fix \varphi).$$ 
	\item If $\lambda \in \RR$ is such that  $\C_\mu(\varphi,F)$ is finite for all $\mu \leq \lambda$, then 
$$\TC_{r,X}(F^\lambda) \leq \sum_{\mu \in (-\infty, \lambda]}\max \{\TC_{r,X}(C) \mid C \in \C_\mu(\varphi,F)\}.$$
\end{enumerate}
\end{theorem*}

We further establish a similar result for sequential parametrized topological complexity, where we need to put the additional assumption that the self-map is fiber-preserving for a certain fibration.
\begin{theorem*}[Theorem \ref{TheoremParamTCLS}]
	Let $r \in \NN$ with $r \geq 2$, let $E$ and $X$ be ANRs and let $p\colon E \to X$ be a fibration with path-connected fibers. We denote $E^r_X= \{(u_1,\dots,u_r) \in E^r \mid p(u_1)=\dots =p(u_r)\}$ and consider 
	$$q_r\colon E^r_X \to X, \quad q_r(u_1,\dots,u_r) = p(u_1).$$
	 Let $\Phi\colon E^r_X \times [0,1] \to E^r_X$ be continuous with $\Phi(u,0)=u$ for all $u \in E^r_X$ and such that
	\begin{equation*}
	q_r(\Phi(u,t))= q_r(u) \quad \forall u \in E^r_X, \ t \in [0,1]. 
	\end{equation*}
Let $F\colon E^r_X \to \RR$ be a Lyapunov function for $\Phi_1:= \Phi(\cdot,1)$ which is bounded from below and such that $(\Phi_1,F)$ satisfies condition (D). Then
	$$\TC_{r,p}(F^\lambda) \leq \sum_{\mu \in (-\infty,\lambda]} \TC_{r,p}(F_\mu \cap \Fix \Phi_1) \qquad \forall \lambda \in \RR.$$ 
\end{theorem*}

As explained above, these general results include the case of critical points of $C^1$-functions. Applying the two theorems to time-1 maps of negative gradient flows, we obtain the following consequences.
\begin{theorem*}
[Corollary \ref{CorTCLSdiff}]
Let $(X,g)$ be a connected and complete Riemannian Banach manifold and let $F\in C^1(X^r)$ be bounded from below. Assume that $(F,g^r)$ satisfies the Palais-Smale condition and that $F(\Crit F)\subset \RR$ is discrete. 
For all $\lambda \in \RR$ it holds that  
$$\TC_{r,X}(F^\lambda) \leq \sum_{\mu \in (-\infty,\lambda]} \max \{\TC_{r,X}(C) \mid C \in \C_\mu (F)\}.$$ 
\end{theorem*}
For the setting of sequential parametrized topological complexity the arguments are more involved, since the gradient flow of an arbitrary function $F\colon E_X^r\to \mathbb{R}$ does not need to respect the fibers of the fibration $q_r\colon E_X^r\to X$.
To overcome this problem we introduce the notion of a \textit{vertically proportional} vector field on $E_X^r$ and show the following.
\begin{theorem*}[Theorem \ref{TheoremParamTCLSdiff}]
Let $E$ and $X$ be paracompact Banach manifolds and let $p\colon E\to X$ be a smooth fibration with path-connected fibers and a surjective submersion. 
Let $F\in C^1(E^r_X)$ be bounded from below. Assume that $(F,\bar{g})$ satisfies the Palais-Smale condition for a suitable complete metric $\bar{g}$ on $E^r_X$ and that $\nabla^{\bar{g}}F$ is vertically proportional with respect to $q_r\colon E^r_X\to X$ and $\bar{g}$. For each $\lambda \in \RR$ it holds that
$$\TC_{r,p}(F^\lambda) \leq \sum_{\mu \in (-\infty,\lambda]} \max \left\{\TC_{r,p}(C)\mid C \in \C_\mu(F) \right\}. $$
\end{theorem*}

We then generalize Farber's notion of navigation function to the sequential and the sequential parametrized setting and apply the above results to the thus-obtained functions.
As a geometric application, we derive the following result from the study of navigation functions in the case of $r=2$. 
\begin{theorem*}[Theorem \ref{TheoremTCApplParallel}]
Let $n \in \NN$, let $M \subset \RR^n$ be a compact hypersurface and let $$\alpha(M):= \# \left\{\{x,y\}\subset M \ \middle| \ x \neq y, \ T_xM=T_yM=(x-y)^\perp \right\}.$$
Then 
$$\alpha(M) \geq \TC(M)-1.$$
In particular, there are at least $\TC(M)-1$ distinct unordered pairs of points $\{x,y\} \subset M$ which satisfy $T_xM =T_yM$.
\end{theorem*}

We will further apply our results on sequential parametrized topological complexity to a framework of equivariant fiber bundles in which the technical assumptions become more tangible. 

\begin{theorem*}[Corollary \ref{CorEquivPTCLS}]
	 Let $G$ be a connected Lie group, let $E$ and $B$ be closed $G$-manifolds, such that the $G$-action on $B$ is transitive, and let $p\colon E \to B$ be a $G$-equivariant fiber bundle. Then for any $G$-invariant $f \in C^1(E)$ it holds that
$$\TC_r[p\colon E\to B] \leq \sum_{\mu \in \RR} \max \Big\{ \TC_{r,p} (\Crit_{\mu_1} f  \times_B \dots \times_B \Crit_{\mu_r} f) \ \Big| \ \mu_1,\dots,\mu_r \in \RR, \, \sum_{i=1}^r \mu_i =\mu\Big\}.$$
\end{theorem*}

As an application of this result, we will obtain the following:

\begin{theorem*}[Theorem \ref{TheoremPTC4m-1}]
Let $m \in \NN$ and let $p\colon U\S^{4m-1}\to \S^{4m-1}$ be the unit tangent bundle of the standard $(4m-1)$-sphere $\S^{4m-1}\subset \RR^{4m}$.
    Then for all $r \in \NN$ with $r\geq 2$ it holds that
    $$    \TC_r[p\colon U\S^{4m-1}\to \S^{4m-1}] = r + 1.       $$
\end{theorem*}

This article is organized as follows. In section 1, we recall the topological approach to LS theory from \cite{RudyakSchlenk}, in particular the notion of an index function,  and its connection to the Palais-Smale condition. We introduce sequential subspace complexities in section 2, establish their index function properties and carry out the general consequences of the LS theorem for sequential topological complexities. The consequences for critical points of $C^1$-functions are carried out in that section as well. These results are then applied to navigation functions in section 3 and to product functions in section 4, yielding the abovementioned consequences and applications.  In section 5 we introduce sequential parametrized subspace complexities and establish a setting in which they form index functions to which the LS-type approach can be applied. Section 6 focuses on the notion of vertical proportionality of a vector field. After introducing this condition, we apply the results of section 5 to differentiable functions with vertically proportional gradients and conclude with a discussion of parametrized navigation functions. In section 7, we will study the parametrized TC of restrictions of fibrations to sublevel sets of functions on their total spaces under the assumption that these are again fibrations and by applying the main results of section 6 to another special class of functions. To obtain examples of fibrations to which this approach applies, we study equivariant fiber bundles over homogeneous spaces in section 8 and establish an upper bound of the parametrized TC of equivariant fiber bundles between closed manifolds with homogeneous codomains. Finally, we apply this result to unit tangent bundles of odd-dimensional spheres and compute their parametrized topological complexities in certain cases. 

\subsection*{Notation} We will use the following conventions throughout the manuscript. Let $X$ be a topological space and let $f\colon X \to \RR$ be a function. 
\begin{itemize}
\item Let $\mathfrak{P}(X)$ denote the power set of $X$.
\item Given a map $\varphi\colon X \to X$, we denote its fixed point set by 
$$\Fix \varphi = \{x \in X \mid \varphi(x)=x\}. $$
\item We denote the \emph{closed} sublevel and level sets of $f$ , resp., by 
$$f^a := \{x \in X \mid f(x) \leq a\}, \quad f_a :=\{x\in X \mid f(x)=a\} \quad \forall a \in \RR.$$

\item If $X$ is a smooth manifold and $f$ is a differentiable function, then we denote 
$$\Crit f = \{x \in X \ | \ Df_x=0\}, \qquad \Crit_\mu f := \Crit f \cap f_{\mu} \quad \forall \mu \in \RR.$$

\item  Given $r \in \NN$ with $r \geq 2$ we put
$$\Delta_{r,X}:= \{(x,x,\dots,x)\in X^r \mid x \in X\}.$$
\item Given $A \subset X$, we call a map $\Phi\colon A \times [0,1] \to X$  a \emph{deformation of $A$} if it is a homotopy from the inclusion $A \hookrightarrow X$ to another map. 
\item Given a vector field $Y\colon M \to TM$ on a manifold $M$, we put $$\sing(Y) := \{x \in M \mid Y(x) =0\}.$$
\end{itemize}

\section*{Acknowledgements}
M. S. was partially funded by the Deutsche Forschungsgemeinschaft (German Research Foundation) -- grant agreement number 518920559.
M. S. is grateful for the support by the Danish National Research Foundation through the Copenhagen Centre for Geometry and Topology (DNRF151).

\section{Topological Lusternik-Schnirelmann theory}

In this section, we want to recall the topological approach to Lusternik-Schnirelmann theory which was carried out by  Rudyak and Schlenk in \cite{RudyakSchlenk}. The basic notions required to develop their approach are summarized in the following definition.\medskip

\begin{definition}
\label{DefIndex}
	Let $X$ be a topological space and let $\varphi\colon X \to X$ be continuous. 	\begin{enumerate}[a)]
		\item An \emph{index function} on $X$ is a map $\nu\colon  \mathfrak{P}(X) \to \NN_0 \cup \{+\infty\}$ which has the following properties: 
\begin{enumerate}[(i)]
	\item (Monotonicity) \enskip If $A \subset B \subset X$, then $\nu(A) \leq \nu(B)$. 
	 	\item (Subadditivity) \enskip For all $A,B \subset X$ it holds that $\nu(A \cup B) \leq \nu(A) + \nu(B).$
	\item (Continuity) \enskip If $A \subset X$ is closed, then there exists an open neighborhood $U$ of $A$ with $\nu(A) = \nu(U)$.
\end{enumerate}

\item  An index function $\nu$ on $X$ is called \emph{$\varphi$-supervariant} if $\nu(\varphi(A))\geq \nu(A)$
for all closed subsets $A \subset X$.
\item A continuous function $f\colon X \to \RR$ is called a \emph{Lyapunov function for $\varphi$} if 
$$f(\varphi(x)) \leq f(x) \qquad \forall x \in X.$$
	\end{enumerate}
\end{definition}

\begin{remark}
\label{RemarkDefSup}
\begin{enumerate}
\item The notion of index function axiomatizes the crucial properties of relative Lusternik-Schnirelmann category. Namely, as shown in \cite[Chapter 1]{CLOT}, given a path-connected $X$ we obtain an index function on $X$ by putting
 $$\nu(A) := \cat_X(A),$$
 i.e. the minimal value of $n \in \NN$, such that $A$ can be covered by the union of $n$ open subsets of $X$ whose inclusions are nullhomotopic.
\item Our definition of an index function is actually a slight simplification of the one from \cite{RudyakSchlenk}, in which index functions are defined relative to a certain subset. A similar notion is introduced under the name of an \emph{abstract category} in \cite[Section 1.7]{CLOT}.
	\item  In \cite{RudyakSchlenk}, the $\varphi$-supervariance of an index function is defined with respect to arbitrary subsets of $X$, not only closed ones. However, a careful check of the proofs in \cite[Section 2]{RudyakSchlenk} shows that it is sufficient to assume supervariance for closed subsets of $X$ to obtain the results that we make use of.
\end{enumerate}
\end{remark}

Rudyak and Schlenk have further introduced a topological condition that mimicks the classical Palais-Smale (PS) condition. Its relation to the original PS condition will be clarified below. 

\begin{definition}[{\cite[Definition 1.1]{RudyakSchlenk}}]
	Let $X$ be a topological space, let $\varphi\colon X \to X$ be continuous and let $f\colon X \to \RR$ be a Lyapunov function for $\varphi$. We say that $(\varphi,f)$ \emph{satisfies condition (D)} if the following two conditions hold:
	\begin{enumerate}
		\item[($D_1$)] $f$ is a \emph{strong} Lyapunov function for $\varphi$, i.e. it holds that
		$$f(\varphi(x)) < f(x) \qquad \forall x \in X \setminus \Fix \varphi.$$
		\item[($D_2$)] If $A \subset X$ is chosen such that $f$ is bounded on $A$ and such that
		$$\inf \{f(a) - f(\varphi(a)) \mid a \in A\}=0,$$
		then $\overline{A} \cap \Fix \varphi \neq \emptyset$.
	\end{enumerate}
\end{definition}

The topological abstraction of the classical Lusternik-Schnirelmann theorem by Rudyak and Schlenk is the following.

\begin{theorem}[{\cite[Theorem 2.3]{RudyakSchlenk}}]
\label{TheoremLSMain}
	Let $X$ be a topological space, let $\varphi\colon X \to X$ be continuous and let $\nu\colon \mathfrak{P}(X) \to \NN\cup \{+\infty\}$ be a $\varphi$-supervariant index function. Assume that there exists a Lyapunov function $f\colon X \to \RR$, which is bounded from below, such that $(\varphi,f)$ satisfies condition (D) and that $f(\Fix \varphi)\subset \RR$ is discrete. Then 
	$$\nu(f^a) \leq \sum_{\lambda \in (-\infty, a]} \nu(f_\lambda \cap \Fix \varphi) \qquad \forall a \in \RR.$$
\end{theorem}

The original version of the Lusternik-Schnirelmann theorem and most of its generalizations can be re-obtained from Theorem \ref{TheoremLSMain} by letting $X$ be a differentiable manifold and choosing $\varphi$ as the time-$1$ map of a negative gradient flow of a $C^1$-function $f\colon X \to \RR$. Then $\Fix \varphi = \Crit f$, so that Theorem \ref{TheoremLSMain} can be used to estimate the number of critical points of a function in one of its sublevels from below. \medskip 

Regarding this setting, we want to recall the PS condition and discuss its relation to condition (D). 
The following definition of a pseudo-gradient is a mild generalization of \cite[Definition II.3.8]{Struwe}.
	
	\begin{definition}
	\label{DefPseudoGrad}
		Let $M$ be a paracompact Banach manifold and let $F \in C^1(M)$. Consider a Riemannian metric $g$ on $M$, let $\nabla^g F\colon M \to TM$ denote the gradient of $F$ with respect to the metric and let $\|\cdot\|$ denote the norm induced by $g$ on any tangent space of $M$. We call a vector field $X \colon M \to TM$ a \emph{pseudo-gradient vector field} or just \emph{pseudo-gradient of $F$} if it has the following properties:
		\begin{enumerate}[(i)]
			\item $X$ is Lipschitz continuous with respect to $g$ and the Sasaki metric on $TM$ induced by $g$, see e.g. \cite[Section X.4]{LangDiff} for a description of the latter.
			\item There exists $C_1>0$, such that $$\|X(x)\| \leq C_1 \cdot  \min \{\|\nabla^gF(x)\|,1\}\qquad \forall x \in M .$$ 
			\item There exists $C_2>0$, such that $$DF_x(X(x)) \geq C_2 \cdot \min \{\|\nabla^gF(x)\|,1\}\cdot \|\nabla^gF(x)\| \qquad \forall x \in M. $$ 
		\end{enumerate}
	\end{definition}
	\begin{remark}
	\label{RemarkFlowCompl}
	Let $(M,g)$ be a Riemannian Banach manifold. 
	\begin{enumerate}
	\item By a fundamental theorem of K. Nomizu and H. Ozeki, see \cite{NomizuOzeki}, a differentiable Banach manifold admits a complete Riemannian metric if and only if it is paracompact. Thus, working with complete Riemannian Banach manifolds implicitly implies that the manifold under consideration is paracompact. 
		\item 	It is shown in \cite[Lemma II.3.9]{Struwe}, see also \cite{PalaisLuster}, that any $F \in C^1(M)$ admits a pseudo-gradient vector field. Moreover, by basic results on ordinary differential equations on manifolds, if $(M,g)$ is complete, then any pseudo-gradient for $F$ admits a complete flow. 
		\item Let $F \in C^1(M)$ and let $X\colon M \to TM$ be a pseudo-gradient of $F$. It follows from properties (ii) and (iii) of Definition \ref{DefPseudoGrad} that $\sing(X) = \Crit F$.
		\end{enumerate}
	\end{remark}

The authors of \cite{RudyakSchlenk} study a particular construction of a pseudo-gradient vector field that we shall briefly elaborate upon. 

\begin{example}
\label{ExamplePG}
Let $(M,g)$ be a complete Riemannian Banach manifold, let $F \in C^1(M)$ and let $\nabla F$ denote the gradient of $F$ with respect to $g$. Consider a smooth and monotonically increasing function $\rho\colon \RR \to \RR$ with the following properties:
\begin{itemize}
	\item $\rho(x)=1$ if $x \leq 1$, 
	\item $\rho(x)\leq x$ if $x \in [1,2]$, 
	\item $\rho(x)=x$ if $x \geq 2$. 
\end{itemize}
Put $h\colon  M \to \RR$, $h(x) := \frac{1}{\rho(\|\nabla F(x)\|)}$. Then $X\colon M \to TM$, $X(x):= h(x) \cdot \nabla F(x)$, is a pseudo-gradient vector field for $F$: property (i) from Definition \ref{DefPseudoGrad} follows from the boundedness of $X$ that is inherent in its definition. Considering property (ii), it follows from our choice of $\rho$ that
$$\|X(x)\|\leq 2 \min \{\|\nabla F(x)\|,1\} \quad \forall x \in M.$$
Concerning property (iii), we derive from $h(x) \geq \min\{\frac{1}{\|\nabla^gF(x)\|},1\}$ that
$$DF_x(X(x)) \geq \min \big\{\|\nabla F(x)\|,\|\nabla F(x)\|^2\big\} \qquad \forall x \in M.$$
Hence, $X$ is a pseudo-gradient vector field for $F$.
\end{example}

In \cite{RudyakSchlenk} it is shown that if $f\colon M \to \RR$ is any function on a complete Riemannian manifold which satisfies the condition (C) of Palais and Smale, see e.g. \cite{PalaisLuster}, then the time-1 map $\phi_1$ of the flow of a negative pseudo-gradient given as in Example \ref{ExamplePG} together with $f$ itself satisfies condition (D). Instead of condition (C), we want to use a different condition in this article, which is more commonly used in studies of gradient flows.

\begin{definition}
Let $(M,g)$ be a complete Riemannian Banach manifold and let $F \in C^1(M)$. We say that $(F,g)$  \emph{satisfies the Palais-Smale condition} or \emph{PS condition} if  every sequence  $(x_n)_{n \in \NN}$ in $M$ for which $(|F(x_n)|)_{n\in\NN}$ is bounded and which satisfies
$$\lim_{n \to \infty} \nabla^gF(x_n)=0$$
has a convergent subsequence.
\end{definition}

The PS condition is slightly stricter than condition (C), see e.g. \cite[Section 3]{MawhinWillem}. We will make occasional use of the following property.

\begin{lemma}
\label{LemmaPSCompLevelsets}
Let $(M,g)$ be a complete Riemannian manifold and let $f \in C^1(M)$, such that $(f,g)$ satisfies the PS condition. Then $f_\mu\cap \Crit f $ has finitely many connected components for each $\mu \in \RR$. 	
\end{lemma}
\begin{proof}
	Assume there exists some $\mu_0 \in \RR$, such that $ N:= f_{\mu_0}\cap \Crit f $ has infinitely many components. Then there exists a sequence $(x_n)_{n \in \NN}$ in $N$, for which no two elements lie in the same component of $N$. Evidently, $(x_n)_{n \in \NN}$ does not have a convergent subsequence. But by definition of $N$, 
	$$f(x_n)=\mu_0, \qquad \nabla^g f(x_n)=0 \qquad \forall n \in \NN, $$
so the PS condition yields that it must have a convergent subsequence. This is a contradiction, so that such a $\mu_0 \in \RR$ can not exist. 
\end{proof}

The following result is a slight variation of \cite[Proposition 9.1]{RudyakSchlenk}. There, the authors study pseudo-gradients of the kind constructed in Example \ref{ExamplePG} as well as condition (C) instead of the PS condition. Their arguments can be transferred easily to the setting considered in this article. For the convenience of the reader, we will carry out some details of the proof.

\begin{theorem}
\label{TheoremPSDPseudo}
	Let $(M,g)$ be a  complete Riemannian Banach manifold, let $F \in C^1(M)$ be bounded from below and assume that $(F,g)$ satisfies the PS condition. Let $X\colon M\to TM$ be a pseudo-gradient vector field of $F$ and let $\phi_1\colon M \to M$ be the time-1-map of the flow of $-X$. Then $F$ is a Lyapunov function for $\phi_1$ and $(\phi_1,F)$ satisfies condition (D).
\end{theorem}

\begin{proof}
Let $\phi\colon  M \times \RR \to M$ be the flow of $-X$ which is well-defined by Remark \ref{RemarkFlowCompl}.(1) and satisfies $\phi(x,1)=\phi_1(x)$ for each $x \in M$. By a standard computation, one derives from property (iii) of Definition \ref{DefPseudoGrad} that $(\phi_1,F)$ satisfies $(D_1)$.
 
Let $A\subset M$ be non-empty, such that $F$ is bounded on $A$ and such that $\inf \{F(a)-F(\phi_1(a)) \mid a \in A\}=0$.  As in the proof of \cite[Proposition 9.1]{RudyakSchlenk}, one obtains sequences $(x_n)_{n \in \NN}$ in $A$ and $(t_n)_{n \in \NN}$ in $[0,1]$, such that 
\begin{equation}
\label{EqDiffPseudo}
 DF_{\phi(x_n,t_n)}(X(\phi(x_n,t_n))) < \frac1n \qquad \forall n \in \NN.
 \end{equation}
 Put $y_n:= \phi(x_n,t_n)$ for each $n$. We derive from \eqref{EqDiffPseudo} and property (iii) of a pseudo-gradient that $\lim_{n \to \infty} \nabla^gF(y_n)=0$.
Let  $C_3 \geq 0$ with $|F(x_n)|\leq C_3$ for all $n \in \NN$. Then $F(y_n)= F(\phi(x_n,t_n)) \leq F(x_n)\leq C_3$ for each $n \in \NN$.
By assumption, $F$ is bounded from below, so it follows that $(F(y_n))_{n \in \NN}$ is bounded. The PS condition for $(F,g)$ thus implies that $(y_n)_{n \in \NN}$ has a subsequence $(y_{n_k})_{k \in \NN}$ which converges to some $z \in \Crit F=\Fix\phi_1$. We want to show that $(x_{n_k})_{k \in \NN}$ converges to $z$ as well. Let $C_1,C_2 \in (0,+\infty)$ be given as in Definition \ref{DefPseudoGrad}. For each $k \in \NN$,  $\gamma_k\colon [0,t_{n_k}] \to M$, $\gamma_k(t) = \phi(x_{n_k},t)$ is a $C^1$-path from $x_{n_k}$ to $y_{n_k}$, so it holds that
\begin{align*}
d_M(x_{n_k},y_{n_k}) &\leq \int_0^{t_{n_k}} \|\gamma'_k(t)\|\, dt = \int_0^{t_{n_k}} \|X(\gamma_k(t))\|\,dt \leq C_1 \int_0^{t_{n_k}} \|\nabla^gF(\gamma_k(t))\|\,dt ,
\end{align*}
 where we used property (ii) of a pseudo-gradient. We further derive from property (iii) that
$\|\nabla^gF(x)\| \leq C_3 \left(DF_x(X(x)) + (DF_x(X(x)))^{\frac12}\right)$ for each $x \in M$, where $C_3:= \max\{C_2^{-1},\sqrt{C_2^{-1}}\}$.
Inserting this into the above computation shows that 
\begin{align*}
d_M(x_{n_k},y_{n_k}) 
&\leq C_1C_3\int_0^1 DF_{\phi(x_{n_k},t)}(X(\phi(x_{n_k},t)))\,dt + C_1C_3 \Big(\int_0^1 DF_{\phi(x_{n_k},t)}(X(\phi(x_{n_k},t)))\, dt\Big)^{\frac12} ,
\end{align*}
where we have used Jensen's inequality for concave functions. We derive from \eqref{EqDiffPseudo} that
$$d_M(x_{n_k},y_{n_k}) \leq C_1C_3\Big(\frac{1}{n_k} + \frac{1}{\sqrt{n_k}}\Big), $$
which yields $\lim_{k \to \infty} d_M(x_{n_k},y_{n_k})=0$. Using the triangle inequality we derive that $(x_{n_k})_{k \in \NN}$ converges to $z$ as well, which shows that $z \in \overline{A} \cap \Fix \phi_1 $. Thus, $(\phi_1,F)$ satisfies ($D_2$).
\end{proof}

\section{Sequential subspace complexities as index functions}
\label{SectionTCrLS}

We want to define an index function via open covers of subsets by domains of continuous sequential motion planners. The latter are defined below and our index functions generalize the concept of relative complexity as introduced in \cite[Section 4.3]{FarberBook}.\medskip

\emph{Throughout this section we fix some $r \in \NN$ with $r \geq 2$.} \medskip 

Given a path-connected topological space $X$, we  consider the path fibration
$$p_r\colon PX \to X^r, \qquad p_r(\gamma) = \big( \gamma(0), \gamma\big(\tfrac{1}{r-1}\big),\dots, \gamma\big(\tfrac{r-2}{r-1}\big),\gamma(1)\big).$$
The  $r$-th sequential topological complexity of $X$, denoted by $\TC_r(X)$, is the sectional category of $p_r$ and was introduced by Rudyak in \cite{RudyakHigher} under the name of higher topological complexity. In the following definition, we introduce a relative version of this notion, which was first investigated by Daundkar in \cite{DaundkarGroup}. 

\begin{definition}Let $X$ be a path-connected topological space and let $A \subset X^r$. 
\label{DefSubspaceTC}
\begin{enumerate}[a)]
	\item  A \emph{sequential motion planner on $A$} is a local section $s\colon  A \to PX$ of $p_r$.
\item The \emph{$r$-th sequential subspace complexity of $A$}, denoted by $\TC_{r,X}(A)$, is given as the minimal value of $n \in \NN$ for which there exist open subsets $U_1,U_2,\dots,U_n \subset X^r$ with $A \subset \bigcup_{i=1}^n U_i$, such that for each $i \in \{1,2,\dots,n\}$ there exists a continuous sequential motion planner $s_i\colon U_i\to PX$. In case there is no such $n$, we put $\TC_{r,X}(A):= \infty$.
\end{enumerate}
\end{definition}
\begin{remark}\label{remark_subspace_vs_absolute_tc} Let $X$ be a path-connected topological space.
\begin{enumerate}
\item Our definition of $\TC_{r,X}(A)$ slightly differs from the one in \cite{DaundkarGroup} as Daundkar considers covers of $A$ by subsets of $A$ which are open in $A$ with respect to the subspace topology. By standard arguments, one shows that both notions coincide if $X$ is an absolute neighborhood retract and if $A$ is closed in $X^r$.
	\item    In terms of Definition \ref{DefSubspaceTC} the $r$-th sequential topological complexity of $X$ is given as $\TC_{r}(X) = \TC_{r,X}(X^r)$. However, we note that if $B\subseteq X$ is a subspace then the $r$-th sequential subspace complexity $\TC_{r,X}(B^r)$ can be different from the intrinsic $r$-th sequential topological complexity $\TC_{r}(B)$.
    This can already be seen in the case $r=2$ taking $X= \S^2$ and $B = \S^1$ where $\S^1$ is embedded as the equator.
    It is well-known that $\TC_2(\S^1) = 2$, but since $\S^1$ is contractible in $\S^2$ we obtain $\TC_{2,\S^2}(\S^1\times \S^1) = 1$.
    \item Sequential topological complexity is motivated by an idea from robot motion planning. Consider a robot working in the configuration space $X$ equipped with a motion planning algorithm, so that for an arbitrary $r$-tuple of points, the robot is able to find a path visiting all $r$ points of the tuple in the given order. $\TC_r(X)$ describes the minimal number of discontinuities of such an algorithm.  If one only wants the robot to find its way for a certain subset of point tuples $A \subset X^r$, but allows it to move inside all of $X$, the number $\TC_{r,X}(A)$ expresses the minimal number of discontinuities of such an algorithm.
    \end{enumerate}
\end{remark}

Our aim in this section is to show that under mild assumptions on $X$, we can define an index function on $X^r$ by putting 
$$\nu(A) = \TC_{r,X}(A)$$
for each $A \subset X^r$. For this purpose,  we need to consider some preparatory technical statements.

The following lemma is an immediate consequence of the assertions of \cite[Remark 3.2]{RudyakHigher} and an analogue of \cite[Lemma 18.1]{FarberSurveyTC}, which has originally been shown in \cite{DaundkarGroup}. For the reader's convenience, we include its statement and proof. 

\begin{lemma}[{\cite[Lemma 2.4]{DaundkarGroup}}]
\label{LemmaTCrhom}
	Let $X$ be a path-connected Hausdorff space, consider $$d\colon X^r \to X^r, \quad d(x_1,x_2,\dots,x_r):=(x_1,x_1,\dots,x_1),$$ and let $A \subset X^r$. The following statements are equivalent:
	\begin{enumerate}
		\item  $A$ admits a continuous sequential motion planner. 
		\item There exists a deformation $h\colon A \times [0,1] \to X^r$ of $A$ with $h(a,1)=d(a)$ for each $a \in A$.
\end{enumerate}
\end{lemma}
\begin{proof}
(1) $\Rightarrow$ (2): \enskip Let $s\colon A \to PX$ be a continuous sequential motion planner. Then 
\begin{align*}
&h\colon A \times [0,1] \to X^r, \quad h(a,t)= \Big(s(a)(0),s(a)\big(\tfrac{1-t}{r-1}\big),,\dots,s(a)\big(\tfrac{(r-2)(1-t)}{r-1}\big),s(a)(1-t)\Big),
\end{align*}
is a deformation with the desired properties. 

 (2) $\Rightarrow$ (1): \enskip Let $h\colon A \times [0,1] \to X^r$ be a deformation of $A$ with $h(a,1)=d(a)$ for all $a \in A$. Denote the components of $h$ by $h(a,t)=(h_0(a,t),\dots,h_{r-1}(a,t))$ for $a \in A$ and $t \in [0,1]$. Define a map $s\colon A \to PX$ by 
 \begin{equation}
 \label{EqTCrhom}
  s(a)(t)= \begin{cases}
	h_{j-1}(a,2((r-1)t-j+1)) & \text{if }  t \in \Big[\frac{j-1}{r-1},\frac{j-\frac12}{r-1}\Big), \\
	h_{j}(a,2(j-(r-1)t))& \text{if }  t \in \Big[\frac{j-\frac12}{r-1},\frac{j}{r-1}\Big],
 \end{cases}
\qquad \forall j \in \{1,2,\dots,r-1\} .
\end{equation}
One checks without difficulties that $s(a)$ is indeed a continuous path with $(p_r \circ s)(a)=a$ for each $a \in A$. Using that $h$ is continuous and $X$ is Hausdorff, one derives from \cite[Theorem VII.2.4]{Bredon} that $s$ is continuous.
\end{proof}

To establish the continuity property of $\nu(A)=\TC_{r,X}(A)$, we need to make an additional assumption on the space under consideration. In the following results, we will assume that the considered space is an \emph{absolute neighborhood retract (ANR)}, see e.g. \cite[p. 287]{CLOT} or \cite{PalaisInf} for a definition. \medskip

The following lemma is an on-the-nose analogue of \cite[Lemma 1.6]{MescherSC} which considers spherical complexities and whose proof is a straightforward analogue of the corresponding statement for Lusternik-Schnirelmann category, see \cite[Lemma 1.11]{CLOT}.

\begin{lemma}
\label{LemmaMPextend}
Let $X$ be a path-connected ANR and let $A \subset X^r$ be a closed subset which admits a continuous sequential motion planner. Then there exists an open neighborhood $U \subset X^r$ of $A$ which admits a continuous sequential motion planner.
\end{lemma}
\begin{proof}
Let again $d\colon X^r \to X^r$, $d(x_1,x_2,\dots,x_r)=(x_1,x_1,\dots,x_1)$. By Lemma \ref{LemmaTCrhom}.a), there exists a deformation $h\colon A \times [0,1] \to X^r$ with $h(a,1)=d(a)$ for each $a \in A$. Put $Q := (X^r \times \{0,1\})\cup (A\times [0,1])$ and define 
\begin{equation}
\label{EqDefH}
H\colon  Q \to X^r, \qquad H(x,t)= \begin{cases}
		x & \text{if } t=0, \\
		d(x) & \text{if } t=1, \\
		h(x,t) & \text{if } x \in A.
	\end{cases}
	\end{equation}
One checks without difficulties that $H$ is well-defined and continuous. Since $X$ is an ANR, $X^r$ is an ANR as well. Hence, there exist an open neighborhood $V\subset X^r \times [0,1]$ of $Q$ and a continuous map $K\colon V \to X^r$ with $K|_Q=H$. Since $[0,1]$ is compact, there exists an open subset $U \subset X^r$ with $ U \times [0,1]\subset V$ and $A \subset U$. By definition of the map, $K|_{U\times [0,1]}$ is a deformation of $U$ with $K(x,1) = d(x)$ for each $x \in U$. Since $X$ is by assumption metrizable, it is a Hausdorff space, so it follows from Lemma \ref{LemmaTCrhom} that $U$ admits a continuous sequential motion planner.
\end{proof}

Using the previous lemma, we can now assert that sequential subspace complexities form index functions.  

\begin{prop}
\label{PropTCrindex}
Let $X$ be path-connected.
	\begin{enumerate}[a)]
		\item For all $A \subset B \subset X^r$, it holds that $\TC_{r,X}(A) \leq \TC_{r,X}(B)$.
  \item For all $A,B \subset X^r$ it holds that $\TC_{r,X}(A\cup B) \leq \TC_{r,X}(A) + \TC_{r,X}(B)$.
  \item Let $A \subset X^r$ be a closed subset. Then there exists an open subset $U\subset X^r$ with $A \subset U$, such that
  $$\TC_{r,X}(A)=\TC_{r,X}(U).$$
  \item Assume that $X$ is normal. Let $A,B \subset X^r$ be closed  with $A \cap B=\emptyset$. Then
$$\TC_{r,X}(A \cup B) = \max\{ \TC_{r,X}(A),\TC_{r,X}(B)\}.$$
	\end{enumerate}
\end{prop}
\begin{proof}	\begin{enumerate}[a)]
\item[a) and b)] These statements follow immediately from the definition of $\TC_{r,X}$.
\item[c)] Let $k:= \TC_{r,X}(A)$. Then by definition of $\TC_{r,X}(A)$ there are $k$ open domains $U_1,\dots,U_k \subset X^r$ of continuous sequential motion planners, such that $A \subset\bigcup_{i=1}^k U_i$. Put $U:= \bigcup_{i=1}^k U_i$. Since $A \subset U$, we derive from a) that $k \leq \TC_{r,X}(U)$. However, since $U_1,\dots,U_k$ form an open cover of $U$ for which each element admits a continuous sequential motion planner, it follows that $\TC_{r,X}(U) \leq k$. 
\item[d)] This follows by standard arguments, see e.g. \cite[p.336]{Fox} for the corresponding result for Lusternik-Schnirelmann category.
\end{enumerate}
\end{proof}

We next want to show that under mild assumptions on $X$ the index function $\TC_{r,X}$ is indeed supervariant with respect to time-1 maps of deformations. 

\begin{prop}
\label{PropTCrdeform}
	Let $X$ be a path-connected ANR, let $A \subset X^r$ be closed and let $\Phi\colon A\times [0,1]\to X^r$ be a deformation of $A$. Put $\Phi_1:=\Phi(\cdot,1)\colon A\to X^r$. If $\Phi_1(A)$ is closed, then
	$$\TC_{r,X}(\Phi_1(A))\geq \TC_{r,X}(A).$$
\end{prop}
\begin{proof}
Put $k:= \TC_{r,X}(\Phi_1(A))$ and let $U_1,\dots,U_k \subset X^r$ be open subsets with $\Phi_1(A)\subset \bigcup_{j=1}^k U_j$, such that each $U_j$ admits a continuous sequential motion planner $s_j\colon U_j \to PX$. Since $X$ is an ANR, $X$ is metrizable, hence normal. As a closed subspace of $X^r$, $\Phi_1(A)$ is normal as well. Thus, see e.g. \cite[Theorem A.1.(6)]{CLOT}, we can find closed subsets $B_j\subset U_j$ for $j \in \{1,2,\dots,k\}$ with $\Phi_1(A)= \bigcup_{j=1}^k B_j$.
Put $B_j' := \Phi_1^{-1}(B_j)$ for each $j \in \{1,2,\dots,k\}$. Then $A \subset \Phi_1^{-1}(\Phi_1(A)) \subset \bigcup_{j=1}^k B_j'$. Moreover, for each $j$ we obtain a continuous motion planner $s_j'\colon B_j'\to PX$ as follows: Denote the components of $\Phi$ by $\phi_j\colon A \times [0,1] \to X$, $j \in \{1,2,\dots,r\}$, such that $\Phi(x,t) = (\phi_1(x,t),\dots,\phi_r(x,t))$. Given $x=(x_0,x_1,\dots,x_{r-1}) \in B_j'$, $\ell \in \{1,2,\dots,r-1\}$ and $t \in \left[\frac{\ell-1}{r-1},\frac{\ell}{r-1}\right]$, we put 
\begin{equation}
\label{Eqsjprime}
\begin{aligned}
s_j'(x)(t):= \begin{cases}
	\phi_{j-1}(x,3(r-1)t-3\ell+3) & \text{if } t \in \Big[\frac{\ell-1}{r-1},\frac{\ell-\frac23}{r-1}\Big), \\
	(s_j(\Phi_1(x)))\left(3t+\frac{1-2\ell}{r-1}\right)	 & \text{if } t \in \Big[\frac{\ell-\frac23}{r-1},\frac{\ell-\frac13}{r-1}\Big), \\
		\phi_j(x,3\ell-3(r-1)t)	 & \text{if } t \in \Big[\frac{\ell-\frac13}{r-1},\frac{\ell}{r-1}\Big]. 
\end{cases} 
\end{aligned}
\end{equation}
One checks that $s_j'$ is continuous and that $(p_r \circ s_j')(x)=x$ for each $x \in B_j'$, so each $s_j'$ is a continuous sequential motion planner. Since $B_j'=\Phi_1^{-1}(B_j)$ is closed and since $X^r$ is an ANR, it follows from Lemma \ref{LemmaMPextend} that there is an open neighborhood $V_j$ of $B_j'$ which admits a continuous sequential motion planner.  Since $A \subset \bigcup_{j=1}^k V_j$, this shows that $\TC_{r,X}(A) \leq k$.
\end{proof}

Note that the condition on $\Phi_1(A)$ to be closed in Proposition \ref{PropTCrdeform} will be satisfied for every closed subset if $\Phi_1$ is a homeomorphism or if $X$ is compact. 

\begin{cor}
\label{CorTCrsupindex}
	Let $X$ be a path-connected ANR and let $\varphi\colon X^r \to X^r$ be a closed map which is homotopic to the identity. Then $\TC_{r,X}\colon \mathfrak{P}(X^r) \to \NN_0\cup \{\infty\}$ is a $\varphi$-supervariant index function.
\end{cor}
\begin{proof}
It follows immediately from Proposition \ref{PropTCrindex} that $\TC_{r,X}\colon \mathfrak{P}(X^r) \to \NN_0\cup \{\infty\}$ is an index function. Let $H\colon X^r \times [0,1] \to X^r$ be a homotopy from $\id_{X^r}$ to $\varphi$. Then $H|_{A \times [0,1]}$ is a deformation of $A$ for every $A \subset X^r$. In particular, if $A$ is closed, it follows from Proposition \ref{PropTCrdeform} that
	$$\TC_{r,X}(\varphi(A))= \TC_{r,X}(H(A \times \{1\})) \geq \TC_{r,X}(A).$$
\end{proof}

After establishing the index function properties, we next want to apply Theorem \ref{TheoremLSMain} to sequential subspace complexities in several settings. Before doing so, let us introduce some shorthand notation that will be useful in the statements of several of the following results. \medskip

 Let $X$ be a topological space, $F\colon X\to \RR$ and $\varphi\colon X \to X$. For each $\mu \in \RR$ we let $\C_\mu(\varphi,F)$ denote the set of connected components of $F_\mu \cap \Fix \varphi$. In the case that $X$ is a smooth Banach manifold and that $F$ is differentiable, let $\C_\mu(F)$ denote the set of connected components of $F_\mu \cap \Crit F$. \medskip

The following result is a generalization of \cite[Theorem 4.32]{FarberBook}.

\begin{theorem}
\label{TheoremTCLS}
	Let $X$ be a path-connected ANR, let $\varphi\colon X^r \to X^r$ be continuous and homotopic to the identity and let $F\colon X^r \to \RR$ be a Lyapunov function for $\varphi$ which is bounded from below. Assume that $F(\Fix \varphi)$ is discrete and that $(\varphi,F)$ satisfies condition (D).
\begin{enumerate}[a)]
\item For all $\lambda \in \RR$, it holds that 
	$$\TC_{r,X}(F^\lambda) \leq \sum_{\mu \in (-\infty,\lambda]} \TC_{r,X}(F_\mu \cap \Fix \varphi).$$ 
	\item If $\lambda \in \RR$ is such that  $\C_\mu(\varphi,F)$ is finite for all $\mu \leq \lambda$, then 
$$\TC_{r,X}(F^\lambda) \leq \sum_{\mu \in (-\infty, \lambda]}\max \{\TC_{r,X}(C) \mid C \in \C_\mu(\varphi,F)\}.$$
\end{enumerate}
\end{theorem}
\begin{proof}
\begin{enumerate}[a)]
	\item This is an immediate consequence of Corollary \ref{CorTCrsupindex} and Theorem \ref{TheoremLSMain}.
	\item Let $\mu \in (-\infty,\lambda]$. Since $X$ is a Hausdorff space, $\Fix \varphi$ is a closed subset of $X^r$. Hence, for every $\mu \in \RR$, the subset $F_{\mu}\cap \Fix \varphi$ is closed in $X^r$ and since every $C \in \C_\mu(\varphi,F)$ is closed in $F_{\mu} \cap \Fix \varphi$, it is closed in $X^r$ as well. Since each $\C_\mu(\varphi,F)$ is finite, we can iteratively apply Proposition \ref{PropTCrindex}.d) and obtain 
$$\TC_{r,X}(F_\mu \cap \Fix \varphi) = \TC_{r,X}\Big(\bigcup_{C \in \C_\mu(\varphi,F)} C\Big) = \max \{\TC_{r,X}(C) \mid C \in \C_\mu(F)\}.$$	
Combining this observation with part a) shows the claim.
\end{enumerate}
\end{proof}

 Since every paracompact Banach manifold is an ANR by \cite[Theorem 5]{PalaisInf} and since every Riemannian Banach manifold is paracompact, we obtain the following as a direct consequence of Theorem \ref{TheoremTCLS}.

\begin{cor}
\label{CorTCLSdiff}
Let $(X,g)$ be a connected and complete Riemannian Banach manifold and let $F\in C^1(X^r)$ be bounded from below. Assume that $(F,g^r)$ satisfies the PS condition, where $g^r$ denotes the product metric of $g$ on $X^r$, and that $F(\Crit F)$ is discrete. Then  
$$\TC_{r,X}(F^\lambda) \leq \sum_{\mu \in (-\infty,\lambda]} \max \{\TC_{r,X}(C) \mid C \in \C_\mu (F)\}\qquad \forall \lambda \in \RR.$$ 
\end{cor}
\begin{proof}
	By \cite[Theorem 4.4]{PalaisLuster}, see also \cite[Lemma II.3.9]{Struwe}, the function $F$ admits a pseudo-gradient $V\colon X^r \to TX^r$ in the sense of Definition \ref{DefPseudoGrad}. Let $\phi_1\colon X^r \to X^r$ be the time-1 map of the flow of $-V$, which is well-defined by the completeness of $(X^r,g^r)$. The flow yields a homotopy from $\phi_1$ to the identity map. By Theorem \ref{TheoremPSDPseudo}, $F$ is a Lyapunov function for $\phi_1$ and $(\phi_1,F)$  satisfies condition (D). Moreover, by Lemma \ref{LemmaPSCompLevelsets}, $\C_\mu(F)$ is finite for all $\mu \in \RR$. Thus, the claim follows directly from Theorem \ref{TheoremTCLS}.b), since by definition of $\phi_1$ and by Remark \ref{RemarkFlowCompl}.(3) it holds that $\Fix \phi_1 = \sing \, V = \Crit F$.
\end{proof}

\section{Navigation functions and sequential subspace complexities} 
\label{SectionNav}

In this section, we want to study consequences of Corollary \ref{CorTCLSdiff} a for special types of differentiable functions, namely $r$-navigation functions. These are straightforward generalizations of the notion of a navigation function as introduced in \cite[Chapter 4]{FarberBook}, see also \cite[Section 4.1]{MescherSurvey} for an exposition.

\begin{definition}
\label{DefNavFunction}
Let $r \in \NN$ with $r \geq 2$ and let $X$ be a connected paracompact Banach manifold. We call a continuously differentiable function $F\colon X^r \to \RR$ an $r$\textit{-navigation function} if 
\begin{enumerate}[(i)]
	\item $F$ satisfies the PS condition with respect to a complete Riemannian metric on $X^r$, 
	\item $F(x)\geq 0$ for all $x \in X^r$ and $F^{-1}(\{0\}) = \Delta_{r,X}$.
	\end{enumerate}
\end{definition}

\begin{remark}
\label{RemarkNavCritzero}Let $r \in \NN$ with $r \geq 2$.
\begin{enumerate}
	\item By definition of an $r$-navigation function $F\colon X^r \to \RR$, its global minimum is $0$ and it holds that $$\Crit_0F = F_0= \Delta_{r,X}.$$
	\item If $X$ is a finite-dimensional closed manifold, then any  $F\in C^1(X^r)$ which satisfies property (ii) of Definition \ref{DefNavFunction} is an $r$-navigation function, since any $C^1$-function on a closed manifold satisfies the PS condition with respect to any Riemannian metric.
	\item The motivation for the introduction of navigation functions is the wish to carry out motion planning along paths following gradient flow lines of a differentiable function. This idea was first established by D. Koditschek and E. Rimon in \cite{KoditschekRimon} who considered motion planning to a fixed end point along gradient flow lines. The idea behind a $2$-navigation function is to use an analogous method for flexible start and end points. This approach is explained in greater detail in \cite[Section 4.1]{MescherSurvey}. 
\end{enumerate}
\end{remark}

It follows from Remark \ref{RemarkNavCritzero}.(1) that if $X$ is not a point, then any $r$-navigation function $F\colon X^r \to \RR$ will have infinitely many critical points, since $\Delta_{r,X} \subset \Crit F$. However, we can use subspace complexities to estimate the numbers of critical points whose value is \emph{positive} as the following application of Corollary \ref{CorTCLSdiff} shows.
\begin{theorem} 
\label{TheoremTCLSNavi}
Let $r \in \NN$ with $r \geq 2$, let $X$ be a connected paracompact Banach manifold and let $F\colon X^r\to \RR$ be an $r$-navigation function. Then 
$$	\TC_{r,X}(F^\lambda) \leq \sum_{\mu \in (0,\lambda]} \max \{ \TC_{r,X}(C) \mid C \in \C_\mu(F)\}+1 \qquad \forall \lambda \in \RR.$$
\end{theorem}
\begin{proof}
	Since $F$ is non-negative, we obtain from Corollary \ref{CorTCLSdiff} and Remark \ref{RemarkNavCritzero}.(1) that 
\begin{align*}
\TC_{r,X}(F^\lambda) &\leq 
\sum_{\mu \in (0,\lambda]}\max \{\TC_{r,X}(C) \mid C \in \C_\mu(F)\}  + \TC_{r,X}(\Delta_{r,X}) \\
&= \sum_{\mu \in (0,\lambda]}\max  \{\TC_{r,X}(C) \mid C \in \C_\mu(F)\} + 1
	\end{align*}
for each $\lambda \in (0,+\infty)$. Here, we used that $s\colon \Delta_{r,X}\to PX$, $s(x,\dots,x) =c_x$, where $c_x$ denotes the constant path in $x$, is a continuous sequential motion planner which extends to an open neighborhood of $\Delta_{r,X}$ by Lemma \ref{LemmaMPextend}, so that  $\TC_{r,X}(\Delta_{r,X})=1$. 
	\end{proof}

Note that Theorem \ref{TheoremTCLSNavi} is a straightforward generalization of \cite[Theorem 4.32]{FarberBook}, which deals with the case of $r=2$ and a smaller class of functions. The following example is discussed in \cite[Example 4.31]{FarberBook} in the case of $r=2$ and under the additional assumption that it is a Morse-Bott function.
\begin{example}
\label{ExampleNavi}
Let $n, r \in \NN$ with $r \geq 2$,  let $M$ be a compact submanifold of $\RR^n$ (without boundary) and let $\|\cdot \|$ denote the Euclidean norm on $\RR^n$. Then
	$$F_r\colon M^r \to \RR, \qquad F_r(x_1,\dots,x_r) = \sum_{i=1}^{r-1} \|x_i-x_{i+1}\|^2 ,$$
	is a smooth function, whose differential is straightforwardly computed as
	\begin{equation}
	\label{EqNavDerivative}	
(DF_r)_{(x_1,\dots,x_r)}(v_1,\dots,v_r) = 2 \sum_{i=1}^{r-1} \left<x_i-x_{i+1},v_i-v_{i+1}\right> 
	\end{equation}
 with respect to the Euclidean inner product.  Evidently, $F$ is non-negative and one checks that 
	$$F_r(x_1,x_2,\dots,x_r) = 0 \quad \Leftrightarrow \quad x_1=x_2=\dots=x_r,$$
	i.e. that $F_r^{-1}(\{0\})=\Delta_{r,M}$. Thus, $F_r$ has property (ii) of Definition \ref{DefNavFunction}. Since $M$ is compact, this yields by Remark \ref{RemarkNavCritzero}.(2) that $F_r$ is an $r$-navigation function. 
	\end{example}
	
Studying $2$-navigation functions of this particular type leads to an interesting geometric consequence for hypersurfaces in Euclidean space. 

\begin{theorem}
\label{TheoremTCApplParallel}
Let $n \in \NN$ and let $M \subset \RR^n$ be a compact hypersurface and let $$\alpha(M):= \# \left\{\{x,y\}\subset M \ \middle| \ x \neq y, \ T_xM=T_yM=(x-y)^\perp \right\}.$$
Then 
$$\alpha(M) \geq \TC(M)-1.$$
In particular, there are at least $\TC(M)-1$ distinct unordered pairs of points $\{x,y\} \subset M$ which satisfy $T_xM =T_yM$.
\end{theorem}
\begin{proof}
	If $\alpha(M)$ is infinite, the statement is evident, so we assume throughout the following that $\alpha(M)$ is finite. Consider the function $F\colon M\times M \to \RR$, $F(x,y)=\|x-y\|^2$, where $\|\cdot \|$ denotes the Euclidean metric on $\RR^n$. Since $M$ is compact, $F$ attains its maximum $\lambda$. As explained in Example \ref{ExampleNavi}, $F$ is a $2$-navigation function, so it follows from Theorem \ref{TheoremTCLSNavi} in the case of $r=2$ that 
\begin{equation}
\label{EqTCpMest}
\TC(M)= \TC(F^\lambda) \leq \sum_{\mu \in (0,\lambda]} \max \{\TC_{2,M}(C) \mid C \in \C_\mu(F)\} + 1.
\end{equation}
A straightforward computation further shows that
	$$DF_{(x,y)}(v,w) = 2\left<x-y,v-w\right> \quad \forall x,y \in M, \, v \in T_xM, \, w \in T_yM,$$
	from which it follows that 
	$$(x,y) \in \Crit F \quad \Leftrightarrow \quad v \in (x-y)^\perp \quad \forall v \in T_xM \cup T_yM.$$
So for each $(x,y) \in \Crit F$ it holds that $T_xM \subset (x-y)^\perp$ and $T_yM \subset (x-y)^\perp$.  Since $M$ is a hypersurface, it holds that $\dim T_xM=\dim T_yM= \dim (x-y)^\perp=n-1$, so that these three spaces have to coincide. In particular, since $F$ is a $2$-navigation function, this shows that
\begin{equation}
\label{EqalphaM}
\alpha(M) = \# \{\{x,y\} \subset M \mid (x,y) \in \Crit F, \  F(x,y) >0\}. 
\end{equation}
Since $\alpha(M)$ is by assumption finite, $F_\mu \cap \Crit F$ is finite for each $\mu>0$. Thus, each $C \in \C_\mu(F)$ consists of a single point, which immediately yields $\TC_{2,M}(C)=1$. Thus, we derive from \eqref{EqTCpMest} that \begin{equation}
\label{EqTCMF}
 \TC(M) \leq \#\{\mu \in (0,\lambda] \mid F_\mu\cap \Crit F\neq \emptyset\} +1.
\end{equation}
While for any $(x,y) \in \Crit F$ we also have $(y,x) \in \Crit F$, both of these points lie in the same level set of $F$. Thus, it follows from \eqref{EqalphaM} and \eqref{EqTCMF} that $ \TC(M) \leq \alpha(M) +1$.
\end{proof}

So far, our guideline was to estimate the number of critical points of functions from below in terms of sequential subspace complexities. In the following, we want to go the opposite way and show in an example how to use $r$-navigation functions to estimates values of sequential TCs from above. \medskip

Let $k,r \in \mathbb{N}$ with $r \geq 2$ and choose numbers $n_1,n_2,\ldots,n_k\in \NN$. The sequential topological complexities of 
\begin{equation}
\label{EqProdSpheres}
   M = \mathbb{S}^{n_1}\times \mathbb{S}^{n_2} \times  \ldots \times \mathbb{S}^{n_k} ,
\end{equation}
were originally computed by Basabe, Gonz\'alez, Rudyak and Tamaki in \cite[Corollary 3.12]{BGRT}. In particular, they have shown that if all of $n_1,n_2,\dots,n_k$ are odd, then 
$$\TC_r(M)=k(r-1)+1.$$ 
They obtained this value as a lower bound of $\TC_r(M)$ by a standard cohomological computation and as an upper bound using an inequality for sequential TCs of product spaces in terms of the sequential TCs of its factors. Here, we want to show that the upper bound can also be obtained by using $r$-navigation functions. \medskip 

For $n \in \NN$ we consider the $n$-sphere as $\mathbb{S}^{n}=\{x\in\RR^{n+1} \mid \|x\|=1\}$, so that for $x \in \S^n$ we have  $T_x\S^{n} = \{v \in \RR^{n+1} \mid \left<x,v\right>=0\}= \{x\}^\perp,$ where $\left<\cdot,\cdot\right>$ and $\| \cdot \|$ again denote the Euclidean inner product and the Euclidean norm, respectively. 
We further consider the $r$-navigation function 
$$  f_n\colon  (\mathbb{S}^{n})^r\to \mathbb{R}, \qquad f_n(x_1,\dots,x_r)= \sum_{i=1}^{r-1} \|x_i-x_{i+1}\|^2 ,$$
from Example \ref{ExampleNavi}. 

\begin{lemma} 
\label{LemmaCritfn}
For each $n \in \NN$ it holds that
$$\Crit f_n = \big\{(x_1,\dots,x_r) \in (\S^n)^r \ \big| \  x_i \in \{x_1,-x_1\} \ \forall i \in \{2,3,\dots,r\}\big\}.$$	
\end{lemma}
\begin{proof}

	Let $x=(x_1,\dots,x_r) \in \Crit f_n$. 	
    Using \eqref{EqNavDerivative}, we first compute for each $v_1 \in T_{x_1}\S^n$ that 
	$$0 = D(f_n)_x(v_1,0,\dots,0) = 2 \left<x_1-x_2,v_1 \right>= - 2 \left<x_2,v_1\right>.$$
Thus, $T_{x_1}\S^n \subset  \{x_2\}^{\perp}$ and since $T_{x_1}\S^n=\{x_1\}^\perp$ it follows for dimensional reasons that $\{x_1\}^\perp=\{x_2\}^\perp$, so $x_1$ and $x_2$ are linearly dependent.
A straightforward inductive argument shows that the set $\{x_1,x_i\}$ is linearly dependent for $i\in \{2,3\dots,r\}$.
Hence, $x_i\in \{x_1,-x_1\}$ for each $i \in \{1,2,\dots,r\}$. 
Conversely, a direct calculation which we omit shows that any such point is indeed a critical point of $f_n$. 
\end{proof}

Let $n\in \NN$ be arbitrary. Using Lemma \ref{LemmaCritfn} one computes for each $(x_1,\dots,x_r) \in \Crit f_n$ that
$$f_n(x_1,\dots,x_r)= 4m, \quad \text{where } \ \ m = \# \{i \in \{2,3,\dots,r\} \mid x_i \neq x_{i-1}\}.$$
The set of critical values of $f_n$ is thus given as
$f_n(\Crit f_n) = \{0,4,\dots,4(r-1)\}.$
For integers $n_1,\dots,n_k \in \NN$ we let again $M = \S^{n_1}\times \S^{n_2} \times \dots \S^{n_k}$ and consider the function 
$$ F\colon M^r \to \RR, \quad    F((x_{11},\ldots,x_{1k}),\ldots ,(x_{r1},\ldots ,x_{rk})) =   \sum_{j=1}^k f_{n_j} (x_{1j},\dots,x_{rj}) .       $$
Since the $f_{n_j}$ are $r$-navigation functions, it follows that $F$ is non-negative with $F^{-1}(\{0\})=\Delta_{r,M}$, so by Remark \ref{RemarkNavCritzero}.(2), $F$ is an $r$-navigation function as well.  
Moreover, it follows that 
{\small $$
\Crit F= \left\{\left((x_{11},\dots,x_{1k}),\dots,(x_{r1},\dots,x_{rk})\right) \in M^r \ \middle| \ x_{ij}\in \{x_{1j},-x_{1j}\} \ \forall i \in \{2,\dots,r\}, \ j \in \{1,\dots,k\}\right\},
$$}
where we have used Lemma \ref{LemmaCritfn}. The critical values of $F$ are further given by
$$F(\Crit F) = \Big\{\sum_{j=1}^k \mu_j \ \Big| \ \mu_j \in f_{n_j}(\Crit f_{n_j}) \ \forall j\in \{1,2,\dots,k\} \Big\}= \{0,4,8,\dots,4k(r-1)\}.$$
Therefore, it follows from  Theorem \ref{TheoremTCLSNavi} that 
\begin{equation}\label{eq_products_of_spheres_sequ_tc_upper_bound}
       \TC_r(M)   \leq \sum_{i=1}^{k(r-1)} \max\{ \TC_{r,M}(\Sigma) \mid \Sigma \in \mathcal{C}_{4i}(F)\} +1. 
\end{equation}

\emph{Assume for the rest of this section that $n_1,\dots,n_k\in \NN$ are odd numbers.} \medskip

Let $i \in \{0,1,\dots,k(r-1)\}$ and let $\Sigma \in \C_{4i}(F)$. By the above description of $\Crit F$, there are numbers $\lambda_{j,\ell} \in \{-1,1\}$, where $j \in \{1,2,\dots,k\}$ and $\ell\in \{2,3,\dots,r\}$, such that
    $$    \Sigma = \{ ((x_1,\ldots,x_k) , (\lambda_{1,2} x_1,\ldots, \lambda_{k,2} x_k),\ldots, (\lambda_{1,r} x_1,\ldots , \lambda_{k,r} x_k)) \mid (x_1,\ldots,x_k) \in M \} .     $$    
 For each $j\in\{1,2,\ldots, k\}$ we further put $\lambda_{j,1}:=1$.
 Choose a nowhere vanishing continuous vector field $X_j$ on $\S^{n_j}$.
By rescaling we can arrange that $\|X_j(x)\|= (r-1)\cdot \pi$ for all $x\in\S^{n_j}$.
   With this normalization, $X_j(x_j)$ is the tangent vector of a geodesic segment $[a,a+\frac{1}{r-1}] \to \S^{n_j}$ from $x_j$ to $-x_j$ for all $a \in \RR$, $x_j \in \S^{n_j}$, $j \in \{1,2,\dots,k\}$. We define a sequential motion planner  $s\colon \Sigma\to PM$ as follows: 
 $$s((x_1,\ldots,x_k) , (\lambda_{1,2} x_1,\ldots, \lambda_{k,2} x_k),\ldots, (\lambda_{1,r} x_1,\ldots , \lambda_{k,r} x_k))=: (\gamma_1,\dots,\gamma_k).$$
Here, for $j \in \{1,2,\dots,k\}$ we let
 $$\gamma_j\colon [0,1] \to \S^{n_j}, \qquad \gamma_j= \gamma_{j1}*\gamma_{j2}*\dots*\gamma_{j(r-1)},$$
 where $*$ denotes the concatenation of paths, and where $\gamma_{j \ell }\colon [\frac{\ell-1}{r-1},\frac{\ell}{r-1}]\to \S^{n_j}$ is for each  $\ell \in \{1,2,\dots,r-1\}$ given as
 \begin{itemize}
 	\item the constant path in $x_j$ if $\lambda_{j,\ell}=\lambda_{j,(\ell+1)}=1$,
 	\item the constant path in $-x_j$ if $\lambda_{j,\ell}=\lambda_{j,(\ell+1)}=-1$,
 	\item the unique length-minimizing geodesic segment from $x_j$ to $-x_j$ with $\gamma'(\frac{\ell-1}{r-1})=X_j(x_j)$ if $\lambda_{j,\ell}=1$ and $\lambda_{j,(\ell+1)}=-1$, 
    \item the unique length-minimizing geodesic segment from $-x_j$ to $x_j$ with $\gamma'(\frac{\ell-1}{r-1})=X_j(-x_j)$ if $\lambda_{j,\ell}=-1$ and $\lambda_{j,(\ell+1)}=1$.
 \end{itemize}
 One checks that the map $s$ we obtain in this way is indeed a continuous sequential motion planner on $\Sigma$. By Lemma \ref{LemmaMPextend} it extends to a continuous sequential motion planner on an open neighborhood of $\Sigma$. Thus, $\TC_{r,M}(\Sigma)=1$ for all  $\Sigma \in \C_{4i}(F)$ and $i \in \{0,1,\dots,k(r-1)\}$.
 Inserting this into \eqref{eq_products_of_spheres_sequ_tc_upper_bound} shows that 
$\TC_r(\S^{n_1}\times \S^{n_2}\times \dots \times \S^{n_k}) \leq k (r-1) + 1$
for all odd numbers $n_1,n_2,\dots,n_k \in \NN$.

\section{Sequential topological complexity of sublevel sets}
\label{SectionProd}

In this section, we want to derive upper bounds of the sequential TC of a sublevel set of a function $f\colon X \to \RR$ from the results of section \ref{SectionTCrLS}. More precisely,  we will apply those results to functions $F\colon X^r \to \RR$ of the form 
$$F(x_1,\dots,x_r) = f(x_1)+f(x_2)+\dots+f(x_r)$$
and will rephrase the upper bounds in terms of $f$. While we have previously  studied subspace complexities of the form $\TC_{r,X}(F^\lambda)$, it turns out that we can find upper bounds of $\TC_r(f^\lambda)$, i.e. for the sequential TCs of $f^\lambda$ and not just for sequential subspace complexities. \medskip 

\textit{Throughout this section let $r \in \NN$ with $r \geq 2$. } \medskip

We begin by making some technical preparations on condition (D) for product functions. 


\begin{lemma}
\label{LemmaCondDsubspace}
Let $X$ be a topological space, let $\varphi\colon X \to X$ be continuous and let $f\colon X \to \RR$ be a Lyapunov function for $\varphi$ which is bounded from below. Consider 
\begin{align*}
&\Phi\colon X^r \to X^r, \quad \Phi(x_1,x_2,\dots,x_r) := (\varphi(x_1),\varphi(x_2),\dots,\varphi(x_r)), \\
 &F\colon X^r \to \RR, \quad \ \  F(x_1,x_2,\dots,x_r) := f(x_1)+f(x_2)+\dots+f(x_r).
\end{align*}
Assume that $(\varphi,f)$ satisfies condition (D). 
\begin{enumerate}[a)]
	\item Then $(\Phi,F)$ satisfies condition (D).
	\item Let $A\subset X^r$ satisfy  $\Phi(A)\subset A$ and 
 \begin{equation}
 \label{EqCondDsubset}
		\overline{A} \cap \Fix \Phi \subset A.
	\end{equation}
	Then $(\Phi|_A,F|_A)$ satisfies condition (D).
\end{enumerate}
\end{lemma}
\begin{proof}
	\begin{enumerate}[a)]
		\item It follows directly from the corresponding property of $(\varphi,f)$ that $(\Phi,F)$ satisfies ($D_1$), so we shall only prove that it satisfies ($D_2$). Let $B \subset X^r$ be such that $F$ is bounded on $B$ and  $\inf\{F(x)-F(\Phi(x))\mid x \in B\}=0$. Let $(x_n)_{n \in \NN}$ be a sequence in $B$ with
 \begin{equation*}
 	\lim_{n \to \infty} (F(x_n)-F(\Phi(x_n)))=0.
 \end{equation*}
We denote the components of the $x_n$ by $x_n = (x_{1,n},\dots,x_{r,n})$ for each $n \in \NN$.
Since
 $$F(x_n)-F(\Phi(x_n))= \sum_{j=1}^r (f(x_{j,n})-f(\varphi(x_{j,n}))) \qquad  \forall n \in \NN$$
 and since each of the summands in the above sum is non-negative, it follows  that
 $$\lim_{n \to \infty} \left(f(x_{j,n})-f(\varphi(x_{j,n}))\right)= 0 \qquad \forall j \in \{1,2,\dots,r\}.$$
Since $(F(x_n))_n$ is bounded and $f$ is bounded from below, it follows that $(f(x_{j,n}))_n$ is bounded for each $j$. By assumption, $(\varphi,f)$ satisfies ($D_2$), so $(x_{1,n})_n$ has a convergent subsequence $(x_{1,n_k})_{n \in \NN}$. Likewise,  $(x_{2,n_k})_k$ has a convergent subsequence $(x_{2,n_{k_\ell}})_{\ell}$. Going on like this until the $r$-th component sequence, we eventually obtain a subsequence of $(x_n)_n$, which we simply denote by $(y_m)_{m \in \NN}$, all of whose component sequences converge. Thus, $(y_m)_{m \in \NN}$ is a convergent subsequence of $(x_n)_{n \in \NN}$ whose limit lies in $\overline{B} \cap \Fix \Phi_1$ by property ($D_1$). This shows that $(\Phi,F)$ has property ($D_2$).
 \item By assumption the restriction $\Phi|_A\colon A \to A$ is well-defined. Since $(\Phi,F)$ satisfies condition ($D_1$) by a), it is evident that $(\Phi|_A,F|_A)$ satisfies ($D_1$). 
 To show $(D_2)$, let $B \subset A$ be a subset, such that $F$ is bounded on $B$ and
$$0=\inf \{F|_A(x)-F|_A(\Phi(x)) \mid x \in B\}=\inf \{F(x)-F(\Phi(x)) \mid x \in B \}.$$ Since $(\Phi,F)$ satisfies ($D_2$), it holds that $\overline{B} \cap \Fix \Phi \neq \emptyset$. We derive from \eqref{EqCondDsubset} that 
$$\overline{B} \cap \Fix \Phi \subset \overline{A} \cap \Fix \Phi \subset A,$$
such that $\cl_A(B)\cap \Fix \Phi = \overline{B} \cap \Fix \Phi \neq \emptyset$, where $\cl_A(B)$ denotes the closure of $B$ in $A$.
Hence, $(\Phi|_A,F|_A)$ satisfies $(D_2)$.
 \end{enumerate}
\end{proof}

\begin{theorem}
\label{TheoremTCLSprod}
		Let $X$ be a path-connected ANR, let $\varphi\colon X\to X$ be continuous and homotopic to the identity, let $f\colon X \to \RR$ be a Lyapunov function for $\varphi$ which is bounded from below and let $\lambda \in \RR$. Assume that $f^\lambda$ is path-connected, that $(\varphi,f)$ satisfies condition (D) and that $f(\Fix \varphi)$ is discrete. 
	For all $\lambda \in \RR$ it holds that
 $$\TC_r(f^\lambda) \leq \sum_{\mu \in (-\infty,r\lambda]} \max \Big\{\TC_{r,f^\lambda}\Big(\prod_{i=1}^r(f_{\mu_i}\cap \Fix \varphi)\Big) \ \Big| \ \mu_1,\dots,\mu_r \in (-\infty,\lambda], \, \sum_{i=1}^r \mu_i=\mu \Big\} $$
 \end{theorem}
\begin{proof}
Choose $\lambda \in \RR$ and put $Y:= f^\lambda \subset X$. 
Since $f$ is a Lyapunov function for $\varphi$, it holds that $\varphi(Y)\subset Y$. Let again $\Phi:=\varphi^r\colon X^r \to X^r$, such that $F$ is a Lyapunov function for $\Phi$. Then $\Phi(Y^r) \subset Y^r$ and since $Y^r$ is closed, we derive from Lemma \ref{LemmaCondDsubspace}.b) that $(\Phi|_{Y^r},F|_{Y^r})$ satisfies condition (D). 
By definition of $F$, it holds that $Y^r \subset F^{r\lambda}$. In terms of the restriction $F|_{Y^r}\colon Y^r\to \RR$, this yields that $Y^r = (F|_{Y^r})^{r\lambda}$, so we derive from Theorem \ref{TheoremTCLS}.a) that
\begin{align*}
\TC_r(Y)&= \TC_{r,Y}(Y^r) = \TC_{r,Y}((F|_{Y^r})^{r\lambda}) \leq \sum_{\mu \in(-\infty,r\lambda]} \TC_{r,Y}((F|_{Y^r})_\mu \cap \Fix (\Phi|_{Y^r})) \\
&= \sum_{\mu \in(-\infty,r\lambda]} \TC_{r,Y}(F_\mu \cap Y^r \cap \Fix \Phi).
\end{align*}
Apparently, since $f$ is bounded from below and since $\Fix \Phi=(\Fix \varphi)^r$, it holds for each $\mu \in \RR$ that
\begin{align*}
F_\mu \cap Y^r\cap \Fix \Phi 
&= \bigcup_{\stackrel{\mu_1,\dots,\mu_r\in (-\infty,r\lambda]}{\mu_1+\dots+\mu_r=\mu}} \prod_{i=1}^r (f_{\mu_i}\cap Y\cap \Fix \varphi) = \bigcup_{\stackrel{\mu_1,\dots,\mu_r\in (-\infty,\lambda]}{\mu_1+\dots+\mu_r=\mu}} \prod_{i=1}^r (f_{\mu_i}\cap \Fix \Phi).
\end{align*}
By assumption on $f$, the latter is a union of finitely many pairwise disjoint closed subsets of $X^r$. Thus, iterating Proposition \ref{PropTCrindex}.d) in the above computation yields
\begin{align*}
	\TC_r(Y) &\leq \sum_{\mu \in (-\infty,r\lambda]} \TC_{r,Y}\Big(\bigcup_{\stackrel{\mu_1,\dots,\mu_r \in (-\infty,\lambda]}{\mu_1+\dots+\mu_r=\mu}} \prod_{i=1}^r (f_{\mu_i} \cap \Fix \varphi) \Big) \\
	&= \sum_{\mu \in (-\infty,r\lambda]} \max\Big\{ \TC_{r,Y}\Big(\prod_{i=1}^r (f_{\mu_i} \cap \Fix \varphi \Big) \ \Big| \ \mu_1,\dots,\mu_r \in (-\infty,\lambda], \ \sum_{i=1}^r\mu_i=\mu \Big\},
\end{align*}
which we wanted to show. 
\end{proof}

\begin{cor}
\label{CorTCAbssubMax}
 Let $(X,g)$ be a complete Riemannian Banach manifold, let $f \in C^1(X)$ be bounded from below and let $\lambda \in \RR$. If $(f,g)$ satisfies the Palais-Smale condition and if $f(\Crit f)$ is discrete, then
\begin{equation*}
\TC_{r}(f^{\lambda}) \leq \sum_{\mu \in (-\infty,r\lambda]} \max \Big\{ \TC_{r,f^\lambda}(\Crit_{\mu_1} f \times \dots \times \Crit_{\mu_r} f) \ \Big|\ \mu_1,\dots,\mu_r \in (-\infty,\lambda], \, \sum_{i=1}^r \mu_i =\mu \Big\}.
\end{equation*}
\end{cor}
\begin{proof}
	This follows from Theorem \ref{TheoremTCLSprod} in the same way that Corollary \ref{CorTCLSdiff}.a) follows from Theorem \ref{TheoremTCLS}.
\end{proof}


\section{Sequential parametrized subspace complexities as index functions}
In the same sense that Rudyak generalized the notion of topological complexity to sequential topological complexities, the notion of parametrized topological complexity as introduced by  Cohen, Farber and Weinberger in \cite{cohen2021topology} was generalized to sequential parametrized topological complexities by Farber and Paul in \cite{FarberPaulI}. It has since been further studied by the same authors in \cite{FarberPaulII} and by Farber and Oprea in \cite{FarberOpreaSeqParam}.  \medskip

\emph{Throughout this section let $r \in \NN$ with $r \geq 2$ and let $p\colon E \to X$ be a fibration with path-connected fibers. We further let $E^I_X$, $E^r_X$ and $\Pi_r$ be defined as in the introduction.}\medskip 

In analogy with the line of argument of section \ref{SectionTCrLS}, we want to construct an index function on $E^r_X$ out of subspace versions of sequential parametrized TCs. 

\begin{definition}
\label{DefParamSubspaceComp}
 \begin{enumerate}[a)]
 \item Let $A \subset E^r_X$. A local section $s\colon A \to E^I_X$ of $\Pi_r$ is called a \emph{sequential parametrized motion planner on $A$} .
\item The \emph{$r$-th sequential parametrized subspace complexity of $A$}, denoted by $\TC_{r,p}(A)$, is given as the minimal value of $n \in \NN$ for which there exist open subsets $U_1,U_2,\dots,U_n \subset E^r_X$ with $A \subset \bigcup_{i=1}^n U_i$, such that there exists a continuous sequential parametrized motion planner  $s_i\colon U_i\to E_X^I$ for each $i \in \{1,2,\dots,n\}$. In case there is no such $n$, we put $\TC_{r,X}(A):= \infty$ .
 \end{enumerate}
\end{definition}

\begin{remark}
\begin{enumerate}[(1)]
	\item Let $B\subset E$ and consider its $r$-fold fiber product $B_X^r \subset E^r_X$.
    Assume that the restriction $p|_{B}\colon B\to X $ is again a fibration.
    Similarly to the unparametrized case, see Remark \ref{remark_subspace_vs_absolute_tc}.(1), we note that one needs to distinguish between the $r$-th sequential parametrized subspace complexity $\TC_{r,p}(B_X^r)$ and the value of 
    $\TC_r[p|_B\colon B\to X]$.
   
    As an example, let $X=M$ be a closed manifold and consider its tangent bundle $p\colon TM\to M$.
    Since the fibers of $p$ are vector spaces, one obtains that $\TC_{r}[p\colon TM\to M] = 1$, see \cite[Proposition 4.5]{cohen2021topology}. 
    Choose a Riemannian metric $g$ on $M$ and let $UM$ be its unit tangent bundle.
    The restriction $p|_{UM}\colon UM\to M$ is a fibration as well.
    Restricting the global sequential parametrized motion planner of the tangent bundle shows that $\TC_{r,p}(UM_M^r) = 1$.
    However, $\TC_{r}[p|_{UM}\colon UM\to M]\geq \TC_r(\mathbb{S}^{\mathrm{dim}(M)-1})$ since its fiber is the sphere $\mathbb{S}^{\mathrm{dim}(M)-1}$, see \cite{FarberPaulI}.
    It is well-known that $\TC_r(\mathbb{S}^{\mathrm{dim}(M)-1}) \geq 2$, so the sequential parametrized subspace complexity differs from the intrinsic sequential parametrized topological complexity of $p|_{B}$.
        \item As elaborated upon in the introductions of \cite{cohen2021topology} and \cite{FarberPaulI}, sequential parametrized topological complexity is motivated by robot motion planning inside a workspace which depends on certain parameters, e.g. the positions of certain obstacles on a warehouse floor. Here, the selection of a parameter corresponds to the selection of a fiber of $p\colon E \to X$, in which motion planning algorithms are carried out. In complete analogy with sequential subspace complexities, see Remark \ref{remark_subspace_vs_absolute_tc}.(2), sequential parametrized subspace complexities correspond to the selection of a subset of $E^r_X$, i.e. a certain set of $r$-tuples of points, such that all points in a tuple lie in the same fiber of $p$, instead of considering motion planning algorithms covering all of $E^r_X$. 
\end{enumerate} 
\end{remark}
Throughout the following we put 
$$ q_r\colon E_X^r \to X,\qquad  q_r(u_1,\dots,u_r) := p(u_1).$$
More abstractly, $q_r$ is given as the pullback fibration
$q_r=\Delta_r^*p^r$,
where $\Delta_r\colon  X \to X^r$, $\Delta(x)=(x,x,\dots,x)$ and where $p^r\colon E^r \to X^r$ denotes the $r$-fold cartesian product of $p$ with itself. In particular, $q_r$ is itself a fibration. \medskip

The following lemma is a straightforward analogue of Lemma \ref{LemmaTCrhom} with respect to parametrized motion planning. 

\begin{lemma}
\label{LemmaPTCrhom}
 Assume that $E$ and $X$ are Hausdorff spaces and let $d_E\colon E^r_X \to E^r_X$, $d_E(u_1,u_2,\dots,u_r)=(u_1,u_1,\dots,u_1)$. 
 Let $A \subset E^r_X$ be a subset. The following statements are equivalent:
\begin{enumerate}
	\item $A$ admits a continuous sequential parametrized motion planner. 
 	\item There exists a deformation $h\colon A \times [0,1] \to E^r_X$ with $h(a,1)=d_E(a)$ for each $a \in A$, such that 
 	$$q_r(h(a,t))=q_r(a) \qquad \forall a \in A, \ t \in [0,1].$$
\end{enumerate} \end{lemma}
\begin{proof}
(1) $\Rightarrow$ (2):\enskip  One constructs a deformation by the same formula as in the proof of the implication ``(1) $\Rightarrow$ (2)'' of Lemma \ref{LemmaTCrhom} and checks that it has the desired properties. \medskip

(2) $\Rightarrow$ (1): \enskip Let $h\colon A \times [0,1] \to E^r_X$ be a fiber-preserving deformation with respect to $q_r$. We further let $h_0,\dots,h_{r-1}\colon A \times [0,1]\to E$ and  $s\colon A\to PE$ be defined as in \eqref{EqTCrhom}.  Since $E$ is Hausdorff, $E^r_X$ is a Hausdorff space, hence $s$ is continuous.  Since $h$ is fiber-preserving with respect to $q_r$, it holds that  $p(h_j(a,t))=p(h_j(a,0))= p(h_0(a,0))$ for all $a \in A$, $t \in [0,1]$ and $j \in \{0,1,\dots,r-1\}$. Consequently
$$p ( (s(a))(t)) = p(h_0(a,0)) \quad \forall a \in A, \ t \in [0,1],$$ 
so  $s(a) \in E^I_X$ for all $a \in A$. Thus, $s$ is a continuous sequential parametrized motion planner. 
\end{proof}

In section 2 we have restricted ourselves to considering ANRs when considering sequential subspace complexities, as we could establish the continuity property of an index function under this assumption. Likewise, we need to make a similar assumption on $p$ to ensure that sequential \emph{parametrized} subspace complexities will have the continuity property. For this purpose, we next introduce an analogue to the definition of an ANR for fibrations that was originally defined by Yagasaki in \cite{Yagasaki} and that can be seen as an ANR-type property for fiberwise spaces.

\begin{definition}
\label{DefANFR}
	Let $Q$ and $B$ be topological spaces and let $f\colon Q \to B$ be continuous. We call $f$ an \emph{absolute neighborhood fiber retract (ANFR)} if $Q$ and $B$ are metrizable and if $f$ satisfies the following condition: \medskip
	
	Let $T$ be a metrizable space, let $q\colon T \to B$ be continuous and let $A\subset T$ be closed. Then for any continuous map $g\colon A \to Q$ with $q|_A=f \circ g$ there are an open neighborhood $U\subset T$ of $A$ and a continuous map $G\colon U \to Q$ with $G|_A=g$ and $q|_U = f \circ G$.
\end{definition}

\begin{remark}
The notion of an ANFR from Definition \ref{DefANFR} slightly differs from the one introduced in \cite{Yagasaki}. In his terminology, our definition is the one of an absolute neighborhood fiber \emph{extensor}. However, Yagasaki has shown in \cite[Proposition 1.1]{Yagasaki} that both notions are indeed equivalent.
\end{remark}

The following result yields a helpful criterion to show the ANFR property for fibrations.

\begin{lemma}[{\cite[1.2.(iii)]{Yagasaki}}]
\label{LemmaANFR}
Let $Q$ and $B$ be topologial spaces and let $f\colon Q \to B$ be a fibration. If $Q$ and $B$ are ANRs, then $f$ is an ANFR.
\end{lemma}

Our concrete aim is to show that $q_r\colon E^r_X \to X$ is an ANFR for which we need the following auxiliary lemma. 

\begin{lemma}
\label{LemmaANRpreim}
Let $Q$ and $B$ be ANRs and let $f\colon Q \to B$ be a fibration. If $C \subset B$ is an ANR, then $f^{-1}(C)$ is an ANR. 	
\end{lemma}
\begin{proof}
Put $D:= f^{-1}(C)$. Let $T$ be a metrizable space, let $A \subset T$ be closed and let $g\colon  A \to D$ be continuous. Since $(f \circ g)(A) \subset C$ and $C$ is an ANR, there exist an open neighborhood $V \subset T$ of $A$ and a continuous map $q\colon V \to B$ with $q|_{A}=f \circ g$ and $q(V) \subset C$. As an open subset of $T$, $V$ is again metrizable. Since $Q$ and $B$ are ANRs, $p$ is an ANFR by Lemma \ref{LemmaANFR}. Thus, there exist a neighborhood $U \subset V$ of $A$, which is open in $T$ and a continuous map $G\colon U \to Q$, such that
$$G|_{A}=g \qquad \text{and} \qquad q|_U= f \circ G.$$
Hence, $(f \circ G)(U) \subset q(U) \subset C$, which shows that $G(U) \subset D$. So we can restrict the codomain of $G$ to $G\colon  U \to D$, which yields a continuous extension of $g$. Hence, $D$ is an ANR. 
\end{proof}

\begin{prop}
\label{PropANR}
If $E$ and $X$ are ANRs, then $E^I_X$ and $E^r_X$ are ANRs.		
\end{prop}
\begin{proof}
	Since $E$ and $X$ are ANRs, $E^r$ and $X^r$ are again ANRs. Moreover, $p^r\colon E^r \to X^r$ is a fibration for which $E_X^r= (p^r)^{-1}(\Delta_{r,X})$. Since $\Delta_{r,X}$ is homeomorphic to $X$, it is an ANR as well. As a special case of Lemma \ref{LemmaANRpreim}, we thus obtain that $E^r_X$ is an ANR. Considering $E^I_X$, we observe that the induced map of path spaces $p_\#\colon  PE \to PX$, $p_\#(\gamma) = p \circ \gamma$,
is a fibration as well and it follows from \cite[Theorem VI.2.4]{HuRetr} that $PE$ and $PX$ are ANRs. By definition of $E^I_X$, it holds that
$E^I_X= p_\#^{-1}(c(X))$, where $c(X) \subset PX$ denotes the space of constant paths in $X$. Since $c(X)$ is homeomorphic to $X$, it is an ANR, so we obtain from Lemma \ref{LemmaANRpreim} that $E^I_X$ is an ANR.
\end{proof}

The following lemma is the analogue of Lemma \ref{LemmaMPextend} for the parametrized setting.

\begin{lemma}
\label{LemmaPMPextend}
Assume that $E$ and $X$ are ANRs and let $A \subset E^r_X$ be a closed subset which admits a continuous sequential parametrized motion planner. Then there exists an open neighborhood $U \subset E^r_X$ of $A$ which admits a continuous sequential parametrized motion planner.
\end{lemma}

\begin{proof}
By Proposition \ref{PropANR}, $E^r_X$ is an ANR, so it follows from Lemma \ref{LemmaANFR} that $q_r\colon E^r_X \to X$ is an ANFR. By Lemma \ref{LemmaPTCrhom}, there exists a deformation $h'\colon A \times [0,1] \to E^r_X$ with $h'(a,1)=d_E(a)$ and $q_r(h'(a,t))=q_r(a)$ for all $a \in A$ and $t \in [0,1]$, where $d_E(u_1,u_2,\dots,u_r)=(u_1,u_1,\dots,u_1)$. We put $Q':= (E^r_X\times \{0,1\}) \cup (A \times [0,1])$ and let $H'\colon Q'\times [0,1] \to E^r_X$ be defined as in \eqref{EqDefH}, only replacing $d$ by $d_E$ in the second case. Then $q_r\circ \pr_1|_Q = q_r \circ H'$, where $\pr_1\colon E^r_X \times [0,1]\to E^r_X$ denotes the projection onto the first factor.  Since $q_r$ is an ANFR, there exist an open neighborhood $V' \subset E_X^r \times [0,1]$ of $Q'$ and a continuous $K'\colon V' \to E_X^r$, such that $K'|_{Q'}=H$ and $q_r \circ \pr_1|_{V'}= q_r \circ K'$. 

 Arguing as in the proof of Lemma \ref{LemmaMPextend}, we obtain an open subset $U' \subset E^r_X$ with $ U' \times [0,1] \subset V'$ and $A \subset U'$. Then $K'|_{U'\times[0,1]}$ is a deformation with $K'(u,1)=d_E(u)$ and $(q_r \circ K')(u,t) = q_r(u)$ for all $u \in U'$ and $t \in [0,1]$. The claim then follows from  Lemma \ref{LemmaPTCrhom}.
\end{proof}
\begin{prop}
\label{PropPTCrindex}. 
	\begin{enumerate}[a)]
		\item For all $A \subset B \subset E^r_X$, it holds that $\TC_{r,p}(A) \leq \TC_{r,p}(B)$. 
 \item For all $A,B \subset E^r_X$ it holds that $\TC_{r,p}(A\cup B) \leq \TC_{r,p}(A) + \TC_{r,p}(B)$.
  \item Let $A \subset E^r_X$ be a closed subset. Then there exists an open subset $U\subset E^r_X$ with $A \subset U$ and
  $$\TC_{r,p}(A)=\TC_{r,p}(U).$$
    \item Assume that $E$ is normal. Let $A,B \subset E^r_X$ be closed  with $A \cap B=\emptyset$. Then
$$\TC_{r,p}(A \cup B) = \max\{ \TC_{r,p}(A),\TC_{r,p}(B)\}.$$
	\end{enumerate}
\end{prop}
\begin{proof}
All parts of the proposition follow in complete analogy with the corresponding parts of Proposition \ref{PropTCrindex}, where we use that by assumption, the fibers of $p$ are path-connected. Here, in d) we use that since $E$ is normal, $E^r$ is normal as well and so is its closed subspace $E^r_X$. 
\end{proof}

\begin{prop}
\label{PropPTCrdeform}
Assume that $E$ and $X$ are ANRs, let $A \subset E^r_X$ be closed and let $\Phi\colon A \times [0,1] \to E^r_X$ be a deformation of $A$ for which $\Phi_1(A)$ is closed and which satisfies 
	\begin{equation}
	\label{EqDeformParam}
	q_r(\Phi(a,t))=q_r(a) \quad \forall a \in A, \ t \in [0,1]	.
	\end{equation}
Then $\TC_{r,p}(\Phi_1(A)) \geq \TC_{r,p}(A)$.
 \end{prop}
 \begin{proof}
 This follows along the lines of the proof of Proposition \ref{PropTCrdeform}, using Lemma \ref{LemmaPMPextend} instead of Lemma \ref{LemmaMPextend}. Note that condition \eqref{EqDeformParam} ensure that the maps $s_j'$ that are given as in \eqref{Eqsjprime}, are indeed sequential parametrized motion planners.  
\end{proof}

\begin{cor}
\label{CorPTCrsupindex}
Assume that $E$ and $X$ are  ANRs and let $\Phi\colon E^r_X \times [0,1] \to E^r_X$ be continuous with
	$$\Phi(u,0)=u, \qquad q_r(\Phi(u,t))=q_r(u) \qquad \forall u \in E^r_X, \ t \in [0,1].$$
	 Then $\TC_{r,p}\colon \mathfrak{P}(E^r_X) \to \NN\cup \{\infty\}$ is a $\Phi_1$-supervariant index function on $E^r_X$, where $\Phi_1(u):=\Phi(u,1)$.
\end{cor}
\begin{proof}
It follows immediately from Proposition \ref{PropPTCrindex}.a)-c) that $\TC_{r,p}$ is an index function. If $A \subset E^r_X$ is closed, then $\Phi|_{A \times [0,1]}$ is a deformation of $A$ and it follows from Proposition \ref{PropPTCrdeform} that $\TC_{r,p}(\Phi_1(A)) \geq \TC_{r,p}(A)$. Thus, $\TC_{r,p}$ is $\Phi_1$-supervariant. 
\end{proof}

Using Corollary \ref{CorPTCrsupindex} we obtain the following result as a direct application of Theorem \ref{TheoremLSMain}.

\begin{theorem}
\label{TheoremParamTCLS}
	Let $E$ and $X$ be ANRs and let $\Phi\colon E^r_X \times [0,1] \to E^r_X$ be continuous with 
	\begin{equation}
	\label{EqParamDeform}
	\Phi(u,0)=u, \quad q_r(\Phi(u,t))= q_r(u) \quad \forall u \in E^r_X, \ t \in [0,1]. 
	\end{equation}
Let $F\colon E^r_X \to \RR$ be a Lyapunov function for $\Phi_1:= \Phi(\cdot,1)$ which is bounded from below and such that $(\Phi_1,F)$ satisfies condition (D). Then
	$$\TC_{r,p}(F^\lambda) \leq \sum_{\mu \in (-\infty,\lambda]} \TC_{r,p}(F_\mu \cap \Fix \Phi_1) \qquad \forall \lambda \in \RR.$$ 
\end{theorem}

\section{Vertically proportional vector fields and parametrized subspace complexities}

We want to study how Theorem \ref{TheoremParamTCLS} can be applied to critical points of differentiable functions on $E^r_X$, assuming that the latter is a smooth manifold. Note that the theorem is not directly applicable to flows of negative pseudo-gradients as such a flow will in general not be fiber-preserving in the sense of condition \eqref{EqParamDeform}. We will overcome this problem by working with flows of vertical components of gradients on the one hand and with a strongly restricted class of functions on the other hand. \medskip

\emph{Throughout this section let $r \in \NN$ with $r \geq 2$, let $E$ and $X$ be paracompact Banach manifolds and let $p\colon E \to X$ be a smooth fibration and a surjective submersion whose fibers are path-connected. }\medskip

Note that $E$ and $X$ are path-connected ANRs by \cite[Cor. to Theorem 5]{PalaisInf}. Since $p^r\colon E^r \to X^r$ is submersive as well, it follows that $E_X^r= (p^{r})^{-1}(\Delta_{r,X})$ is a smooth Banach manifold whose tangent space at $u=(u_1,\dots,u_r)\in E^r_X$ is given by
\begin{equation}
\label{EqTErX}
T_uE^r_X = \{(v_1,v_2\dots,v_r) \in T_u E^r \mid Dp_{u_1}(v_1)=Dp_{u_2}(v_2)=\dots = Dp_{u_r}(v_r)\}.
\end{equation}
We recall that given a smooth fibration $f\colon M \to N$ the vertical tangent space of $f$ at a point $m \in M$ is given by 
$$\Ver_m(f):= \ker Df_m.$$
In particular, $\Ver_m(f)$ is always a closed linear subspace of $T_mM$. Evidently, it holds for each $u=(u_1,u_2,\dots,u_r) \in E^r_X$ that 
$$\Ver_u(p^r) = \ker (Dp^r)_u = \prod_{i=1}^r \ker Dp_{u_i} = \prod_{i=1}^r \Ver_{u_i}(p).$$
Let again $q_r\colon E^r_X \to X$, $q_r(u_1,\dots,u_r)=p(u_1)$, which is a smooth fibration and a surjective submersion as well. Moreover, for each $u=(u_1,\dots,u_r) \in E^r_X$ it holds that 
\begin{align*}
\Ver_u(q_r)&= \{v \in T_uE^r_X \mid (Dq_r)_u(v)=0\} \\
&= \{(v_1,\dots,v_r) \in T_uE^r \mid 0= Dp_{u_1}(v_1)=Dp_{u_2}(v_2)=\dots = Dp_{u_r}(v_r)\} = \Ver_u(p^r).
\end{align*}
Thus, we have shown that 
\begin{equation}
	\label{EqVeruq}
	\Ver_u(q_r) = \prod_{i=1}^r \Ver_{u_i}(p) \qquad \forall u=(u_1,u_2,\dots,u_r) \in E^r_X.
\end{equation}

The following result is an immediate consequence of the results of the previous section. We recall that a vector field $Y\colon M \to TM$ on a smooth manifold $M$ is called \emph{forward-complete} if it admits a flow on all of $M \times [0,+\infty)$. 

\begin{theorem}
\label{TheoremParamTCLSvf}
	Let $V\colon E^r_X \to TE^r_X$ be a forward-complete continuous vector field with 
	$$V(u) \in \Ver_u(p^r) \qquad \forall u \in E^r_X.$$
	Assume that the time-1 map $\Phi_1$ of its flow admits a Lyapunov function $F\colon E^r_X \to \RR$, which is bounded from below, such that $(\Phi_1,F)$ satisfies condition (D). Then
	$$\TC_{r,p}(F^\lambda) \leq \sum_{\mu \in (-\infty,\lambda]} \TC_{r,p}(F_\mu \cap \mathrm{sing}(V)) \qquad \forall \lambda \in \RR.$$
\end{theorem}
\begin{proof}
Let $\Phi\colon E^r_X\times [0,+\infty) \to E^r_X$ be the flow of $V$. By assumption and the above computations, it holds that $V(u) \in \Ver_u(q_r)$ for all $u \in E^r_X$ for each $u \in E^r_X$. By elementary results for flows of vector fields, this yields 
$$q_r(\Phi(u,t))=q_r(u) \quad \forall u \in E^r_X, \ t \in [0,+\infty).$$
Thus, since $\Phi_1(u)=\Phi(u,1)$ for each $u \in E^r_X$, we derive from Theorem \ref{TheoremParamTCLS} that 
$$\TC_{r,p}(F^\lambda) \leq \sum_{\mu \in (-\infty, \lambda]} \TC_{r,p}(F_\mu \cap \Fix \Phi_1).$$
Since by construction $\Fix \Phi_1=\sing(V)$, the claim immediately follows.
\end{proof}

\begin{definition}
\label{DefVertProp}
	Let $M$ and $N$ be Banach manifolds, let $q\colon M \to N$ be smooth and let $g$ be a Riemannian metric on $M$.
	\begin{enumerate}[a)]
		\item Let $\pr_\ver\colon  TM \to TM$ be the bundle map for which $\pr_\ver|_{T_mM}$ is the orthogonal projection onto $\Ver_m(q)$ with respect to $g$ for each $m \in M$.
\item Given a vector field $Y\colon M \to TM$, we define another vector field by 
 $$Y^\ver\colon M \to TM, \quad Y^\ver := \pr_{\ver} \circ Y.$$
  Given $F \in C^1(M)$, we further put $\nabla^\ver F:= (\nabla^{g}F)^\ver$,  
where $\nabla^{g}F\colon M \to TM$ denotes the gradient of $F$ with respect to $g$. 
\item We call a vector field $Y\colon  M \to TM$ \emph{vertically proportional} with respect to $q$ and $g$ if there exists $C >0$, such that
	$$\|Y(x) \|_{g}\leq C \cdot \|Y^\ver(x)\|_{g} \qquad \forall x \in M,$$
	where $\|\cdot\|_g$ denotes the norm induced by $g$ on each $T_xM$.
	\end{enumerate}
\end{definition}

\begin{remark}
\label{RemarkSingVertProp}
	Let $M$, $N$, $g$ and $q$ be given as in Definition \ref{DefVertProp} and let $Y\colon M \to TM$ be a vector field. Apparently, $\sing(Y) \subset \sing(Y^\ver)$. Conversely, if $Y$ is vertically proportional, then  every singular point of $Y^\ver$ is a singular point of $Y$ as well. Thus, if $Y$ is vertically proportional, then it necessarily holds that
\begin{equation}
\label{EqSingVertProp}
\sing(Y) = \sing(Y^\ver). 
\end{equation}
\end{remark}

As we will see in the next results,  the existence of a vertical proportional gradient of a function $F \in C^1(E^r_X)$ will allow us to apply Theorem \ref{TheoremParamTCLSvf} to estimate the number of critical points of $F$ by sequential parametrized subspace complexities. However, as condition \eqref{EqSingVertProp} already indicates, the vertical proportionality of its gradient is a strong restriction on the function itself.

\begin{prop}
\label{PropVertPseudoGrad}
	Let $(M,g)$ be a complete Riemannian Banach manifold, let $N$ be a smooth Banach manifold and let $q \colon M \to N$ be smooth. Let $F \in C^1(M)$,  let $\rho\colon \RR \to \RR$  and $h\colon M \to \RR$ be given as in Example \ref{ExamplePG} and put 
	$$ Y\colon M \to TM, \quad Y(m) := h(m)\cdot \nabla^{g}F(m).$$
Assume $\nabla^{g}F$ is vertically proportional with respect to $q$ and $g$.
\begin{enumerate}[a)]
	\item Then $Y^\ver\colon M \to TM$ is a pseudo-gradient vector field for $F$.
	\item Let $\phi^\ver_1\colon M \to M$ be the time-1 map of the flow of $-Y^\ver$. Then $(\phi^\ver_1,F)$ satisfies condition (D).
\end{enumerate}
\end{prop}

\begin{proof} 
As discussed in Example \ref{ExamplePG}, $Y$ is a pseudo-gradient for $F$. For each $m \in M$ we let $\|\cdot \|$ denote the norm on $T_mM$ induced by $g$. By assumption, there exists $C \in (0,+\infty)$, such that $\|\nabla^g F(m)\|\leq C \cdot \|\nabla^\ver F(m)\|$ for all $m \in M$. By definition of $Y$, this immediately yields that 
$$\|Y(m)\| \leq C \cdot \|Y^\ver(m)\| \qquad \forall m \in M.$$
\begin{enumerate}[a)]
	\item  Let us check the three properties of a pseudo-gradient for $Y^\ver$ as listed in Definition \ref{DefPseudoGrad}.
	\begin{enumerate}[(i)]
		\item Since $\pr_\ver$ is fiberwise given by the bounded-linear projection onto $\Ver_u(q)$, it follows from the construction of the Sasaki metric that $\pr_\ver$ is  $1$-Lipschitz-continuous. Thus, it follows from the Lipschitz continuity of $Y$ that $Y^\ver$ is Lipschitz-continuous.
		\item Evidently, $\|Y^\ver(m)\| \leq \|Y(m)\|$ for each $m \in M$, so the upper bound of $\|Y^\ver(m)\|$ follows immediately from the corresponding upper bound of $\|Y\|$, which exists since $Y$ is a pseudo-gradient. 
		\item One easily sees that 
$Y^\ver (m) = h(m)\cdot \nabla^\ver F(m)$ for each $m \in M$.
		Thus, 
		\begin{align*}
		DF_m(Y^\ver(m))&= h(m) \cdot DF_m(\nabla^\ver F(m))= h(m) \cdot g_M(\nabla F(m),\nabla^\ver F(m)) \\
		&= h(m) \cdot \|\nabla^\ver F(m)\|^2\geq h(m) \cdot C^{-2} \|\nabla^{g} F(m)\|^2
		\end{align*}
		for each $m \in M$. One further computes that 
		$$DF_m(Y^\ver(m)) \geq C^{-2} \|\nabla^{g}F(m)\|= C^{-2} \min \{\|\nabla^{g}F(m)\|,\|\nabla^{g}F(m)\|^2\} \qquad \forall m \in M,$$
which shows property (iii) of a pseudo-gradient if we put $C_2 := C^{-2}$ in Definition \ref{DefPseudoGrad}. 
	\end{enumerate}
	\item This follows immediately from combining part a) with Theorem \ref{TheoremPSDPseudo}.
\end{enumerate}
\end{proof}

\emph{Throughout the rest of this section, we fix a complete Riemannian metric $g$ on $E$ and let $\bar{g}$ be the Riemannian metric on $E^r_X$ that is obtained by restricting the product metric $g^r$ on $E^r$.} \medskip 

Using Proposition \ref{PropVertPseudoGrad}, we can now apply Theorem \ref{TheoremParamTCLSvf} to critical points of functions and thereby establish an analogue to Corollary \ref{CorTCLSdiff} for the parametrized setting.

\begin{theorem}
\label{TheoremParamTCLSdiff}
Let $F\in C^1(E^r_X)$ be bounded from below. Assume that $(F,\bar{g})$ satisfies the PS condition and that $\nabla^{\bar{g}}F$ is vertically proportional with respect to $q_r\colon E^r_X\to X$ and $\bar{g}$. For each $\lambda \in \RR$ it holds that
$$\TC_{r,p}(F^\lambda) \leq \sum_{\mu \in (-\infty,\lambda]} \max \left\{\TC_{r,p}(C)\mid C \in \C_\mu(F) \right\}. $$
\end{theorem}
\begin{proof}
Let $Y\colon E^r_X \to TE^r_X$ be the pseudo-gradient of $F$ that is constructed from $\nabla^{\bar{g}}F$ as in Proposition \ref{PropVertPseudoGrad}. By definition of $Y^\ver$, it holds that 
$$\sing(Y^\ver) = \sing (\nabla^\ver F) \stackrel{\eqref{EqSingVertProp}}{=} \sing (\nabla^{\bar{g}}F)= \Crit F.$$
 By Proposition \ref{PropVertPseudoGrad}, we can apply Theorem \ref{TheoremParamTCLS} to $Y^\ver$ with respect to $q_r$ and thus obtain for each $\lambda \in \RR$ that
$$\TC_{r,p}(F^\lambda) \leq \sum_{\mu \in (-\infty,\lambda]}\TC_{r,p}(F_\mu \cap \Crit F). $$
Using Proposition \ref{PropPTCrindex}.d), the claim is then derived in the same way as Corollary \ref{CorTCLSdiff}.   
\end{proof}

 In the rest of this section, we will apply this result to parametrized analogues of navigation functions, which are defined as follows. .

 \begin{definition}
 \label{DefParamNavi}
We call a continuously differentiable function $F\colon E^r_X \to \RR$ a \emph{parametrized $r$-navigation function} if 
\begin{enumerate}[(i)]
	\item $F$ satisfies the PS condition with respect to a complete Riemannian metric on $E^r_X$, 
	\item $F(u)\geq 0$ for all $u \in E^r_X$ and $F^{-1}(\{0\}) = \Delta_{r,E}$.
	\end{enumerate}
\end{definition}

\begin{remark}
\label{RemarkParamNavi}
\begin{enumerate}
	\item Evidently, if $F\colon E^r \to \RR$ is an $r$-navigation function in the sense of Definition \ref{DefNavFunction}, then $F|_{E^r_X}$ will be a parametrized $r$-navigation function.
	\item In analogy with Remark \ref{RemarkNavCritzero}.(2), one argues that if $E$ and $X$ are closed finite-dimensional manifolds, then any continuously differentiable $F\colon E^r_X\to \RR$ which has property (ii) of Definition \ref{DefParamNavi} will be a parametrized $r$-navigation function.
	\end{enumerate}
\end{remark}

Without additional difficulties, one can now establish an analogue of Theorem \ref{TheoremTCLSNavi}.

\begin{theorem}
\label{TheoremParamTCLSNavi} Let $F \in C^1(E^r_X)$ be a parametrized $r$-navigation function for which $\nabla^{\bar{g}}F$ is vertically proportional with respect to $q_r$ and $\bar{g}$. Then 
$$	\TC_{r,p}(F^\lambda) \leq \sum_{\mu \in (0,\lambda]} \max \{ \TC_{r,p}(C) \mid C \in \C_\mu(F)\}+1 \qquad \forall \lambda \in \RR.$$
\end{theorem}
\begin{proof}
	This follows from Theorem \ref{TheoremParamTCLSdiff} in the same way as Theorem \ref{TheoremTCLSNavi} is derived from Corollary \ref{CorTCLSdiff}. The sequential motion planner by constant paths on $\Delta_{r,E}$ is indeed a sequential parametrized motion planner which extends to an open neighborhood of $\Delta_{r,E}$ by Lemma \ref{LemmaPMPextend}.
\end{proof}
 
In the following, we present a method of obtaining parametrized navigation functions which satisfy the requirements of Theorem \ref{TheoremParamTCLSNavi}. \medskip

Let $E$ and $B$ be closed smooth manifolds and let $p \colon E \to B$ be a smooth fibration. Since a smooth fibration between finite-dimensional manifolds is necessarily a submersion, $E\times_B E$ is a smooth manifold as well. For suitable numbers $N_1,N_2 \in \NN$ we choose smooth embeddings $i_B\colon B \hookrightarrow \RR^{N_1}$ and $i_E\colon E \hookrightarrow \RR^{N_2}$. Then one easily checks that
$$i_0\colon  E \times_B E \to \RR^{N_1+2N_2}, \qquad i_0(e_1,e_2) = (i_B(p(e_1)),i_E(e_1),i_E(e_2)),$$
is a smooth embedding as well.

Recall that if $\pi\colon M\to N$ is a fibration between smooth manifolds, where $M$ is equipped with a Riemannian metric we define the \textit{horizontal subspace at} $x\in M$ by
$$     \mathrm{Hor}_x(\pi) :=   \big( \mathrm{Ver}_x(q)\big)^{\perp} \subset T_x M .       $$

\begin{prop}
\label{PropParamNaviEucl}
	Consider the function  $$F\colon  E \times_B E \to \RR, \qquad F(e_1,e_2) = \|i_E(e_1)-i_E(e_2)\|^2,$$
	where $\| \cdot \|$ denotes the Euclidean norm on $\RR^{N_1}$.
	\begin{enumerate}[a)]
		\item Then $F$ is a parametrized $2$-navigation function.
		\item Let $g_0$ be the Riemannian metric on $E \times_B E$ for which $i_0$ is an isometric embedding with respect to the Euclidean metric on $\RR^{N_1+2N_2}$. 
        Assume that for all $(e_1,e_2)\in E\times_B E$ and $u\in \mathrm{Hor}_{(e_1,e_2)}(q)$ we have 
\begin{equation}
\label{EqParamOrth}
        \langle i_E (e_1) - i_E(e_2), u_1-u_2 \rangle = 0    
        \end{equation}
        where we write $(Di_0)_{(e_1,e_2)}(u) = (u_0,u_1,u_2)\in \RR^{n_1}\times \RR^{2n_2}$.
      Then $\nabla^{g_0} F$ is vertically proportional with respect to $q\colon E \times_BE \to B$, $q(e_1,e_2)=p(e_1)$.
	\end{enumerate}
\end{prop}
\begin{proof}
    \begin{enumerate}[a)]
	\item  Apparently, $F$ is smooth with $F(e_1,e_2)\geq 0$ for all $(e_1,e_2) \in E \times_BE$ and $F(e_1,e_2)=0$ holds if and only if $e_1=e_2$, i.e. if $(e_1,e_2) \in \Delta_{2,E}$. Since $E \times_B E$ is compact, this yields that $F$ is a navigation function, see Remark \ref{RemarkParamNavi}.(2).
	\item  
    To compute $\nabla^{g_0}F$, we consider the function $\varphi\colon  \RR^{N_2} \times \RR^{N_2} \to \RR$, $\varphi(x,y)=\|x-y\|^2$, which is easily seen to be differentiable with 
	$$D\varphi_{(x,y)}(v,w) = 2\left<x-y,v-w\right> 	\qquad \forall x,y,v,w \in \RR^{N_2}$$
     with respect to the Euclidean inner product. With $\pr_2\colon \RR^{N_1} \times \RR^{2N_2}\to \RR^{2N_2}$ denoting the projection, we observe that 
    $F = \varphi \circ \pr_2\circ  i_0$ 
    and thus 
    \begin{equation}\label{eq_description_of_differential}
          DF_{(e_1,e_2)} = D\varphi_{(i_E(e_1),i_E(e_2))}\circ (D\pr_2)_{i_0(e_1,e_2)}\circ (Di_0)_{(e_1,e_2)}
    \end{equation}
    Let $(e_1,e_2)\in E\times_B E$ and $u\in \mathrm{Hor}_{(e_1,e_2)}(q)$. We write $(Di_0)_{(e_1,e_2)} (u) = (u_0,u_1,u_2)\in \RR^{n_1+2n_2}$.     Then equation \eqref{eq_description_of_differential} yields 
    $$    DF_{(e_1,e_2)} (u) =   D\varphi_{(x_1,x_2)} (u_1,u_2) =  2\langle  i_E(e_1) - i_E(e_2), u_1 - u_2 \rangle   ,  $$
    which vanishes by assumption.
    It follows that $\nabla^{g_0} F(e_1,e_2) \in \Ver_{(e_1,e_2)}(q)$ for all $(e_1,e_2) \in E \times_B E$. We immediately derive that $\nabla^{g_0} F(e_1,e_2)=\nabla^\ver F(e_1,e_2)$ for all $(e_1,e_2) \in E \times_BE$. In particular, $\nabla^{g_0} F$ is vertically proportional.
    \end{enumerate}
\end{proof}

\begin{remark}\label{remark_condition_for_vertical_proportionality}
Condition \eqref{EqParamOrth} in Proposition \ref{PropParamNaviEucl}.b) is satisfied if $(Di_0)_{(e_1,e_2)}(u)\in \RR^{n_1}\times \Delta_{\RR^{n_2}}$ for all $(e_1,e_2)\in E\times_B E$ and $u\in \mathrm{Hor}_{(e_1,e_2)}(q)$, i.e. if $(Di_0)_{(e_1,e_2)}(u) = (u_0,u_1,u_1)$ for some $u_0 \in \RR^{n_1}$ and some $u_1 \in \RR^{n_2}$.
\end{remark}

\begin{example}
    Let $k,n\in\NN$ with $2\leq k \leq n-1$.
    Consider the Stiefel manifold of orthonormal $k$-frames in $\RR^n$
    $$      V_k(\RR^n) = \{ X\in M_{n,k}(\RR) \,|\, X^T X = 1_k\}     $$
    where $M_{n,k}(\RR)$ denotes the set of real $n\times k$-matrices and $1_k\in M_{k,k}(\RR)$ is the identity matrix.
     For $X\in M_{n,k}(\RR)$ we write $X= (x_1,\ldots, x_k)$ where $x_1,\ldots, x_k\in \RR^n$ are the columns of $X$. 
    We see that $V_k(\RR^n)$ is an embedded submanifold of $M_{n,k}(\RR) \cong \RR^{nk}$ and we endow $V_k(\mathbb{R}^n)$ with the Riemannian metric induced by the Euclidean metric on $\RR^{nk}$.
    Then   $$p\colon V_k(\RR^n)\to \mathbb{S}^{n-1}, \qquad    p(  X)  =  x_1 \quad \text{for}\ \ X = (x_1,\ldots, x_k)\in V_k(\RR^n) ,      $$ is a fiber bundle with typical fiber $V_{k-1}(\RR^{n})$.
    Proposition \ref{PropParamNaviEucl} yields a parametrized $2$-navigation function $F\colon V_k(\RR^n)\times_{\S^{n-1}} V_k(\RR^n) \to \RR$ given by
    $$    F( X,Y) =   \| X-Y \|^2 \quad \text{for }\quad (X,Y)\in V_k(\RR^n) \times_{\S^{n-1}} V_k(\RR^n) .   $$
    Here $\|\cdot\|$ denotes the Euclidean norm on $M_{n,k}(\RR)\cong \RR^{nk}$.
For  $X=(x_1,\ldots,x_k),Y=(y_1,\ldots,y_k)\in V_k(\RR^n)$ with $p(X) = p(Y)$, it holds that $x_1 = y_1$.
    Thus, we see that
    $$   F(X,Y) =   \sum_{j=2}^k \| x_j-y_j\|^2  \qquad \forall (X,Y)\in V_k(\RR^n)\times_{\S^{n-1}}V_k(\RR^n) .     $$
    A computation of the horizontal space of the bundle $V_k(\mathbb{R}^n)\times_{\mathbb{S}^{n-1}} V_k(\mathbb{R}^n)\to \mathbb{S}^{n-1}$ shows that we are in the situation of Remark \ref{remark_condition_for_vertical_proportionality} and the condition of Proposition \ref{PropParamNaviEucl}.b) is satisfied in this example.
    The gradient vector field $\nabla^{g_0}F$ is thus vertically proportional. 
\end{example}

\section{Sequential parametrized TC of restrictions of fibrations to sublevel sets}

We want to apply the reults from the previous section in the case that a smooth fibration $p\colon E \to X$ between Banach manifolds and a function $f \in C^1(E)$ are given, such that the restriction $p|_{f^\lambda}\colon f^\lambda \to X$ is again a fibration. We note that this is a strong assumption, but we will establish a class of fibrations having this property in section \ref{SectionEqui} below. 

Here, we want to derive upper bounds of $\TC[p|_{f^\lambda}\colon f^\lambda \to X]$ under the fibration assumption, i.e. upper bounds of sequential parametrized TCs instead of sequential parametrized subspace complexities. Such a result can be seen as a ``parametrized analogue'' of Corollary \ref{CorTCAbssubMax} on sublevel sets. 

For this purpose, we will study functions $F\in C^1(E^r_X)$ for which there exists $f \in C^1(E)$, such that $$F(u_1,\dots,u_r)=f(u_1)+\dots + f(u_r).$$

\emph{Throughout this section let $r \in \NN$ with $r \geq 2$, let $E$ and $X$ be paracompact Banach manifolds and let $p\colon E \to X$ be a smooth fibration and a surjective submersion whose fibers are path-connected. We further let $g$ be a Riemannian metric on $E$ and let $\bar{g}$ be the Riemannian metric on $E^r_X$ that is obtained by restricting the product metric $g^r$ on $E^r$. Given $A_1,\dots,A_r \subset E$, we will further denote}
$$A_1 \times_X A_2 \times_X \dots \times_X A_r := \Big\{(a_1,\dots,a_r) \in \prod_{i=1}^r A_i \ \Big| \  p(a_1)=\dots=p(a_r)\Big\} = E^r_X \cap \prod_{i=1}^r A_i . $$

As a first step, we want to investigate the critical points of such functions as well as the vertical proportionality of their gradients. We first introduce another condition on the critical points of a function in a general setting. 

\begin{definition}
Let $M$ and $N$ be Banach manifolds, let $q\colon M \to N$ be a smooth submersion and let $M_x:= q^{-1}(\{x\})$ denote the fiber of $q$ over $x \in N$. We say that a function $F \in C^1(M)$ \emph{satisfies condition (A)} if the following is satisfied:\medskip 
\begin{center}
(A) \enskip Let $x \in N$ and let $u \in M_x$ be a critical point of $F|_{M_x}$. Then $u$ is a critical point of $F$. \medskip
\end{center}

\end{definition}

\begin{remark}
\label{RemarkcondA}
Let $M$ and $N$ be Banach manifolds, $\rho$ be a Riemannian metric on $M$, $q:M \to N$ be a smooth submersion and $F \in C^1(M)$. One checks without difficulties that $\nabla^\ver F(u)$ coincides with the gradient of $F|_{M_{q(u)}}$ at $u\in M$ with respect to $\rho|_{M_{q(u)}}$. Thus, $F$ satisfies condition (A) if and only if for each $u \in M$ it holds that 
$$ \nabla^\ver F(u) =0 \quad \Rightarrow \quad \nabla^\rho F(u)=0.$$ 
From this observation it is evident that if $\nabla^\rho F$ is vertically proportional, then $F$ necessarily satisfies condition (A).
\end{remark}

The reason for studying condition (A) in our context is that it allows us to understand the critical points of functions on $E^r_X$ that arise as restrictions of products in a better way and without studying any Riemannian metric. 

\begin{prop}
\label{PropCritF}
	Let $f \in C^1(E)$ and consider $F\colon E^r_X \to \RR$, $F(u_1,\dots,u_r) := \sum_{i=1}^r f(u_i)$. If $f$ satisfies condition (A), then 
	$$\Crit F = \Crit f \times_X \Crit f \times_X \dots \times_X \Crit f.$$
\end{prop}
\begin{proof}
	The inclusion ``$\supset$'' is evident, so it only remains to show that every critical point of $F$ arises in this way. Let $u=(u_1,\dots,u_r) \in \Crit F$, such that
	$$DF_u(v_1,\dots,v_r) = \sum_{i=1}^rDf_{u_i}(v_i) = 0 \qquad \forall (v_1,\dots,v_r) \in T_uE^r_X.$$
Let $X_i := p^{-1}(\{p(u_i)\})$ for each $i \in \{1,2,\dots,r\}$. Since $(0,\dots,0,v_i,0,\dots,0) \in T_u E^r_X$ for all $v_i \in \Ver_{u_i}(p)$, it particularly follows for each $i$ that 
$$Df_{u_i}(v_i) =0 \quad \forall v_i \in \Ver_{u_i}(p) \quad \Rightarrow u_i \in \Crit (f|_{X_i}).$$
Since $f$ satisfies condition (A), this yields that $u_i \in \Crit f$ for each $i \in \{1,2,\dots,r\}$, which shows the claim. 
\end{proof}

Next we want to investigate the vertical proportionality of gradients of such functions with respect to $q_r$. For this purpose we establish an auxiliary lemma.

\begin{lemma}
\label{LemmaVertProd}
Let $f \in C^1(E)$ and consider $F\colon E^r_X \to \RR$, $F(u_1,\dots,u_r) := \sum_{i=1}^r f(u_i)$. Then 
$$\nabla^{\ver}F(u) = \left(\nabla^\ver f(u_1),\dots,\nabla^\ver f(u_r) \right) \quad \forall u = (u_1,\dots,u_r) \in E^r_X,$$
where $\nabla^\ver F$ is considered with respect to $q_r$ and $\bar{g}$, while $\nabla^\ver f$ is considered with respect to $p$ and $g$.  
\end{lemma}

\begin{proof}
Fix $u=(u_1,\dots,u_r) \in E^r_X$. For each $i \in \{1,2,\dots,r\}$ we put $$E_i := \{0\}^{i-1} \times \Ver_{u_i}(p) \times \{0\}^{r-i}\subset T_uE^r_X.$$
Taking orthogonal complements with respect to $\bar{g}|_{T_uE^r_X}$ and $g|_{T_{u_i}E}$, we observe that 
\begin{align*}
\left(\Ver_u(q_r)\right)^\perp &= \bigcap_{i=1}^r E_i^\perp = \bigcap_{i=1}^r \Big(T_uE^r_X \cap \Big(\prod_{j=1}^{i-1} T_{u_j}E \times (\Ver_{u_i}(p))^\perp \times \prod_{j=i+1}^r T_{u_j}E\Big) \Big)\\
&= T_uE^r_X \cap \prod_{i=1}^r (\Ver_{u_i}(p))^\perp.
\end{align*}
Let $\pr_{q_r}\colon T_uE^r_X \to T_uE^r_X$ and $\pr_{i}\colon T_{u_i}E \to T_{u_i}E$, $i \in \{1,2,\dots,r\}$, denote the orthogonal projections onto $\Ver_u(q_r)$ and $\Ver_{u_i}(p)$, respectively, with respect to the chosen metrics.  By definition and the above computation,
$$\ker \pr_{q_r}=T_uE^r_X \cap \prod_{i=1}^r (\Ver_{u_i}(p))^\perp = T_uE^r_X \cap \prod_{i=1}^r \ker \pr_i = \ker \Big(\Big(\prod_{i=1}^r \pr_i\Big)\Big|_{TE^r_X}\Big).$$
We further observe that equation \eqref{EqVeruq} can be rephrased by saying that the images of $\pr_{q_r}$ and $(\prod_{i=1}^r \pr_i)|_{TE^r_X}$ coincide. Together with the equality of their kernels, this shows that
\begin{equation}
\label{EqProjections}
\pr_{q_r}= \Big(\prod_{i=1}^r \pr_i\Big)\Big|_{TE^r_X}. 
\end{equation}
As orthogonal projections, the maps $\pr_{q_r}$ and $\pr_i$, $i \in \{1,2,\dots,r\}$, are self-adjoint with respect to the chosen metrics, so we obtain for all $v=(v_1,\dots,v_r) \in T_uE^r_X$ that
\begin{align*}
\bar{g}\left(\nabla^\ver F(u),v \right)&= \bar{g}\left(\pr_{q_r}(\nabla^{\bar{g}} F(u)),v \right)=\bar{g}\left(\nabla^{\bar{g}} F(u),\pr_{q_r}(v) \right)\\
&= DF_u(\pr_{q_r}(v))\stackrel{\eqref{EqProjections}}=  \sum_{i=1}^r g(\nabla^gf(u_i),\pr_i(v_i)) =\sum_{i=1}^r g(\pr_i(\nabla^gf(u_i)),v_i) \\
&= \sum_{i=1}^r g(\nabla^\ver f(u_i),v_i) = \bar{g} \left((\nabla^\ver f(u_1),\dots,\nabla^\ver f(u_r)) ,v\right).
\end{align*}
Since $v$ was chosen arbitrarily, the claim immediately follows. 
\end{proof}

\begin{lemma}
\label{LemmaVertPropProd}
Let $f \in C^1(E)$ and consider $F\colon E^r_X \to \RR$, $F(u_1,\dots,u_r) = \sum_{i=1}^r f(u_i)$. 	If $\nabla^g f\colon E \to TE$ is vertically proportional with respect to $p$ and $g$, then $\nabla^{\bar{g}} F\colon E^r_X \to TE^r_X$ is vertically proportional with respect to $q_r$ and $\bar{g}$. 
\end{lemma}
\begin{proof} 
Since $\nabla^g f$ is vertically proportional, there exists $C>0$ with $\|\nabla^g f(e)\|_g\leq C\cdot \|\nabla^\ver f(e)\|_g$ for all $e \in E$. By definition, $F= f^r \circ i$, where $i\colon E^r_X \hookrightarrow E^r$ denotes the embedding, and one checks without difficulties that
$$\nabla^{\bar{g}}F(u) = (Di_u)^*(\nabla^{g^r} f^r(u))= (Di_u)^*(\nabla^g f(u_1),\dots,\nabla^g f(u_r))$$
for all $u=(u_1,\dots,u_r)\in E^r_X$, where $(Di_u)^*$ denotes the adjoint operator of $Di_u:T_uE^r_X \to T_{i(u)}E^r$ with respect to $\bar{g}$ and $g^r$. Thus, we compute for all $u=(u_1,\dots,u_r)$ that 
\begin{align*}
	\|\nabla^{\bar{g}}F(u)\|_{\bar{g}}&= \|(\nabla^g f(u_1),\dots,\nabla^g f(u_r))\|_{g^r}=\sqrt{\sum_{i=1}^r\|\nabla^g f(u_i)\|_g^2} \\
&\leq \sqrt{C^2\sum_{i=1}^r \|\nabla^\ver f(u_i)\|_g^2} = C \cdot \|\nabla^\ver F(u_1,\dots,u_r)\|_{\bar{g}},
\end{align*}
where we used Lemma \ref{LemmaVertProd} for the final equality.
\end{proof}

If we apply Theorem \ref{TheoremParamTCLS} to a function as considered in the previous lemmas, we obtain the following result.

\begin{theorem}
\label{TheoremParamTCLSprodrel}
Let $f \in C^1(E)$ be bounded from below and consider $$F\colon E^r_X \to \RR, \qquad F(u_1,\dots,u_r) =f(u_1)+\dots+ f(u_r).$$ Assume that $(E,g)$ is complete, that $(f,g)$ satisfies the PS condition, that $f(\Crit f)$ is discrete in $\RR$ and that $\nabla^gf$ is vertically proportional with respect to $p$ and $g$. Then for all $\lambda \in \RR$ it holds that
$$\TC_{r,p}(F^{\lambda}) \leq \sum_{\mu \in (-\infty,\lambda]} \max \left\{ \TC_{r,p} (\Crit_{\mu_1} f \times_X  \dots \times_X \Crit_{\mu_r} f ) \mid \mu_1+\dots+\mu_r=\mu\right\}.$$
\end{theorem}
\begin{proof}
	Since $(f,g)$ satisfies the PS condition, $(f^r,g^r)$ satisfies the PS condition on $E^r$. Since $E^r_X$ is a closed submanifold of $E^r$, it follows that $F=f^r|_{E^r_X}$ satisfies the PS condition as well. Moreover, $\nabla^{\bar{g}} F$ is vertically proportional with respect to $q_r$ by Lemma \ref{LemmaVertProd}. Thus, it follows from Theorem \ref{TheoremParamTCLSdiff}.a) that
	$$\TC_{r,p}(F^\lambda) \leq \sum_{\mu \in (-\infty,\lambda]} \TC_{r,p}(F_\mu \cap \Crit F).$$
Since $\nabla^g f$ is vertically proportional, $f$ satisfies condition (A) and it follows from Proposition \ref{PropCritF} that
\begin{align*}
\Crit_\mu F 
&= \bigcup_{\mu_1+\mu_2+\dots+\mu_r=\mu} \Crit_{\mu_1} f\times_X \Crit_{\mu_2} f \times_X \dots \times_X \Crit_{\mu_r} f.
\end{align*}
Since $f$ is continuously differentiable, $\Crit_\mu f$ is closed in $E$ for each $\mu \in \RR$, which yields that the latter is a union of closed subsets of $E^r_X$. One derives from the discreteness of $f(\Crit f)$ that the above is a finite union of pairwise disjoint closed subsets. Thus, by iteratively applying Proposition \ref{PropPTCrindex}.d), we obtain that 
\begin{align*}
\TC_{r,p}(\Crit_\mu F) &= \TC_{r,p} \Big(\bigcup_{\mu_1+\mu_2+\dots+\mu_r=\mu} \Crit_{\mu_1}(f)\times_X \Crit_{\mu_2} f \times_X \dots \times_X \Crit_{\mu_r} f \Big)	\\
&= \max \{\TC_{r,p}(\Crit_{\mu_1} f \times_X \Crit_{\mu_2} f \times_X \dots \times_X \Crit_{\mu_r} f) \mid \mu_1+\dots+\mu_r=\mu\}
\end{align*}
for all $\mu \in \RR$. The claim immediately follows.
\end{proof}

We recall that in Theorem \ref{TheoremTCLSprod} we have established an upper bound of sequential TCs of sublevel sets by sequential subspace complexities that are \emph{not} taken with respect to the whole domain of the function, but only with respect to the sublevel set itself. We want to establish an analogue of this result for sequential parametrized TCs and restrictions of fibrations of the form $p_\lambda := p|_{f^\lambda}\colon f^\lambda \to X$. However, $p_\lambda$ will in general not be a fibration and will only have this property in very particular settings. 
We refer the reader to the next section for examples of situations in which this assumption is indeed satisfied.

\begin{theorem}
\label{TheoremParamTCLSprod}
	Let $f \in C^1(E)$ be bounded from below and consider the function $$F\colon E^r_X \to \RR, \qquad F(u_1,\dots,u_r) =f(u_1)+\dots+ f(u_r).$$ 
	Assume that $(E,g)$ is complete, that $(f,g)$ satisfies the Palais-Smale condition and that $f(\Crit f)$ is discrete in $\RR$. If $\lambda \in \RR$ is chosen such that $f^\lambda$ is path connected, that 
	$$p_\lambda := p|_{f^\lambda}\colon  f^\lambda \to X$$
	is a smooth fibration and that $\nabla^gf|_{f^{<\lambda+\delta}}$ is vertically proportional with respect to $p$ for some $\delta >0$, then 
	\begin{align*}
	&\TC_r[p_\lambda\colon f^\lambda \to X] \\
	&\leq \sum_{\mu \in (-\infty,r\lambda]} \max \left\{ \TC_{r,p_\lambda} (\Crit_{\mu_1} f  \times_X \dots \times_X \Crit_{\mu_r} f) \mid \mu_1,\dots,\mu_r \in (-\infty, \lambda], \, \mu_1+\dots+\mu_r=\mu\right\}.
	\end{align*}
\end{theorem}
\begin{proof}
 Let $Y_f\colon E \to TE$ and $Y_F\colon E^r_X \to TE^r_X$ be the pseudo-gradients of $f$ and $F$, respectively, that are constructed as in Proposition \ref{PropVertPseudoGrad}.a). Let $\varphi^\ver\colon E \times \RR \to E$ and $\Phi^\ver\colon E^r_X \times \RR \to E^r_X$ denote the flows of $Y_f^\ver$ and $Y_F^\ver$, respectively. One derives from Lemmas \ref{LemmaVertProd} and \ref{LemmaVertPropProd} that 
\begin{equation}
\label{EqPhivarphi}
\Phi^\ver(u_1,\dots,u_r,t) = (\varphi^\ver(u_1,t),\dots,\varphi^\ver(u_r,t))\qquad \forall (u_1,\dots,u_r) \in E^r_X, \, t \in \RR.
\end{equation}
 Since $f(\Crit f)$ is discrete, we can assume w.l.o.g. that $(\lambda,\lambda+\delta]$ does not contain a critical value of $f$. Put $U:= (f^{<\lambda+\delta})^r \cap E^r_X$. One derives from \eqref{EqPhivarphi} that  $\Phi^\ver(u,t) \in U$ for all $u \in U$ and $t \in [0,+\infty)$, such that $\Phi^\ver$ restricts to a flow	$\Psi:=\Phi^\ver|_{U \times [0,+\infty)} \colon U \times [0,+\infty) \to U$.
	In other words, $Y_F^\ver|_U$ is forward complete. By Proposition \ref{PropVertPseudoGrad}.b), $(\varphi^\ver_1,f)$ satisfies condition (D). Since $\partial U \subset (f_{\lambda+\delta})^r$ and since $\lambda+\delta$ is a regular value of $f$, it holds that  $\partial U \cap \Fix \Phi=\emptyset$. Thus, by Lemma \ref{LemmaCondDsubspace}.b) and \eqref{EqPhivarphi}, $(\Psi_1,F|_U)$ satisfies condition (D). 
	
	Since $F(u) < r\lambda$ for each $u \in U$, it holds that $U = (F|_U)^{r\lambda}$. Since $p_\lambda$ is a fibration, we can thus apply Theorem \ref{TheoremParamTCLSdiff} to $\Psi_1$ and $F|_U$ and obtain that
\begin{align*}
\TC_r[p_\lambda\colon f^\lambda \to B] &= \TC_{r,p_\lambda} ((f^\lambda)^r_X)= \TC_{r,p_\lambda}\left((F|_U)^{r\lambda}\right) \leq \sum_{\mu \in (-\infty,r\lambda]} \TC_{r,p_\lambda} (F_\mu \cap \Crit F).
\end{align*}
From here on, the claim follows along the lines of the proof of Theorem \ref{TheoremParamTCLSprodrel}.
\end{proof}

\section{Equivariant fiber bundles and vertical proportionality}
\label{SectionEqui}

In contrast to the setting of sequential TCs in section 2 to 4, our Lusternik-Schnirelmann-type results for sequential \emph{parametrized} setting in the last sections had the downside that some restrictive conditions have to be assumed, namely 
\begin{itemize}
	\item the vertical proportionality of gradients in Theorem \ref{TheoremParamTCLSdiff} and all of its consequences,
	\item the fibration property of restrictions of fibrations to sublevel sets in Theorem \ref{TheoremParamTCLSprod}.
\end{itemize}  
In this section, we will study certain classes of examples of $G$-equivariant fiber bundles, where $G$ is a Lie group, in which these properties hold or can be derived from more tangible criteria. \medskip

\emph{Throughout this section we let $G$ be a connected finite-dimensional Lie group, $E$ and $B$ be smooth paracompact Banach manifolds and $p\colon E \to B$ be a smooth fiber bundle whose fibers are path-connected. We further assume that $E$ and $B$ are equipped with smooth left $G$-actions}
$$ \Phi\colon  G \times E \to E, \qquad \varphi\colon G \times B \to B,$$
\emph{such that $p$ is $G$-equivariant with respect to these actions. Given a subset $A \subset B$ we further put $E|_A := p^{-1}(A)$ and denote by $E_b := E|_{\{b\}}=p^{-1}(\{b\})$ the fiber over $b$ for each $b \in B$. }

\begin{prop}
\label{PropEquivFiberProps}
Assume that the $G$-action $\varphi$ on $B$ is transitive and let $f\in C^1(E)$ be $G$-invariant. 
\begin{enumerate}[a)]
	\item For each $a \in \RR$, the restriction 
$$p_a\colon f^a \to B, \qquad p_a := p|_{f^a},$$
is a fiber bundle with typical fiber $E_b \cap f^a$ for some $b \in B$. 
\item Each $b \in B$ admits an open neighborhood $U \subset B$ and a local trivialization $\psi\colon E|_U \to U \times E_b$ of $p$, such that
 $$(f \circ \psi^{-1})(b,x)=f(x) \qquad \forall b \in U,\ x \in E_b.$$ 
\end{enumerate}
\end{prop}
\begin{proof}
\begin{enumerate}[a)]
	\item Choose $b_0\in B$ and put $F:= E_{b_0} \cap f^{a}$. Since $\varphi$ acts transitively on $B$,     
    the map $$\pi_{\varphi} \colon  G\to B, \qquad \pi_\varphi(g)=\varphi(g,b_0),$$ is a principal $G$-bundle, so there is an open neighborhood $U$ of $b_0$ which admits a local section $s\colon U\to G$ of $\pi_{\varphi}$. We consider the map
    $$  \psi\colon p_a^{-1}(U) \to U\times F   , \qquad \psi(x) :=  (p(x),  \Phi( s(p(x))^{-1} , x)) .     $$
   Note that $p_a^{-1}(U) = E|_U \cap f^a$.  One checks that indeed $ \Phi( s(p(x))^{-1} , x) \in F$ for all $x \in p_a^{-1}(U)$. 
    Moreover, the map $\psi$ is clearly continuous and one checks that a continuous inverse to $\psi$ is given by $\psi^{-1}(b,x) = \Phi(s(b),x)$, so that $\psi$ is indeed a local trivialization of $p_a$. Since $b$ was chosen arbitrarily, this shows that $p_a$ is a fiber bundle.
    \item Let $b_0$, $U$ and $\psi$ be given as in the proof of a). Then $f(\psi^{-1}(b,x)) = f(\Phi(s(b),x))=f(x)$ for all $b \in U$ by the $G$-invariance of $f$. 
    \end{enumerate}
\end{proof}

In Proposition \ref{PropEquivFiberProps}.a) we have established a class of fiber bundles with the property that all of its restrictions to sublevel sets of a $G$-invariant function are again fiber bundles. To check in which of these settings we can apply Theorem \ref{TheoremParamTCLSprod}, we next want to study the vertical proportionality of gradients of such functions. \medskip

\emph{Throughout the following, we put }
\begin{align*}
&\Phi_g\colon E \to E, \quad \Phi_g(x):= \Phi(g,x) \quad \forall g \in G, \qquad \Phi^x\colon  G \to E, \quad \Phi^x(g) := \Phi(g,x) \quad \forall x \in E.	
\end{align*}
\emph{We further}
\begin{itemize}
	\item \emph{assume that $E$ admits and is equipped with a $\Phi$-invariant Riemannian metric, which we shall fiberwise denote by $\left<\cdot,\cdot\right>_x$,  $x \in E$,}
	\item \emph{assume that $B$ admits and is equipped with a Riemannian metric, for which $p\colon E \to B$ becomes a Riemannian submersion,}
	\item \emph{assume that $G$ admits and is equipped with a bi-invariant Riemannian metric.}
\end{itemize}

The fiberwise norms on the tangent spaces of $E$ and $G$, respectively, induced by these metrics shall both be denoted by $\|\cdot \|$. We further denote the gradient of  $f\in C^1(E)$ with respect to $\left<\cdot,\cdot\right>$ by $\nabla f$ and let its vertical part $\nabla^\ver f$ be defined with respect to $p$ and $\left<\cdot,\cdot\right>$. We further let $\| \cdot \|_\op$ denote the operator map of a linear map between normed vector spaces whenever it is apparent which normed spaces we are referring to. 

\begin{remark}
By a classical result of Milnor, see \cite[Lemma 7.5]{MilnorLie}, a Lie group admits a bi-invariant metric if and only if it is isomorphic as a Lie group to a group of the form $K \times \RR^m$, where $K$ is a compact Lie group and $m \in \NN_0$. Thus, all Lie groups that we consider subsequently are of this form. 
\end{remark}

We want to establish a tangible criterion on a $C^1$-function on $f$ which implies the vertical proportionality of its gradient, for which we introduce some additional terminology. 
\begin{definition}
\label{Deffcontr}
	Let $f\in C^1(E)$ be bounded from below. We say that the $G$-action $\Phi$ is \emph{$f$-controlled} if there exists a function $C\colon [\inf f,+\infty) \to (0,+\infty)$, such that
\begin{equation}
\label{Eqfcontr}
	\|(D\Phi^x)_g \|_\op \leq C(f(x)) \quad \forall g \in G, \ x \in E.
	\end{equation}
	\end{definition}
	
	\begin{remark}
	\label{Remarkfcontr}
	As one easily sees, if $G$ is compact and if $f \in C^1(E)$ is such that $f^\lambda$ is compact for each $\lambda \in \RR$, then $\Phi$ will be $f$-controlled. This holds in particular for any $f \in C^1(E)$ if both $G$ and $E$ are compact. 
	\end{remark}

\begin{lemma}
\label{LemmaEstvertprop}
Let $f \in C^1(E)$ be $G$-invariant and bounded from below. Assume that $\varphi$ is a transitive $G$-action on $B$ and  that $\Phi$ is $f$-controlled with respect to $C\colon [\inf f,+\infty) \to (0,+\infty)$. Let $\lambda \in \RR$, let $U \subset B$ be open and let $b_0 \in B$. Assume that there is a smooth local section $s\colon U \to G$ of $\pi_\varphi\colon G \to B$, $\pi_\varphi(g)=\varphi(g,b_0)$.   Then 
	$$\|\nabla f(x)\| \leq (C(\lambda)\cdot \|Ds_{p(x)}\|_{\op}+1)\cdot \|\nabla^\ver f(x)\| \qquad \forall x \in f^\lambda \cap E|_U.$$
\end{lemma}

\begin{proof}
 Since $f$ is $G$-invariant, $\nabla f$ is $G$-equivariant, i.e. it holds that
\begin{equation*}
\nabla f(\Phi_g(x)) = (D\Phi_g)_x(\nabla f(x))\quad \forall g \in G, \ x \in E. 
\end{equation*}
Since the metric on $E$ is $G$-invariant and $p$ is $G$-equivariant, the differential of $\Phi_g$ maps vertical tangent spaces with respect to $p$ isomorphically onto one another for each $g \in G$. Thus, 
\begin{equation}
\label{EqverGradPhig}
\nabla^\ver f(\Phi_g(x)) = (D\Phi_g)_x(\nabla^\ver f(x)) \quad \forall g \in G, \ x \in E.
\end{equation}
Put $\eta\colon E|_U \to G$,  $\eta(x) := (s(p(x)))^{-1}$, and $h\colon E|_U \to E$, $h(x) :=\Phi(\eta(x),x)$. Since $f$ is $G$-invariant, it holds that $f(h(x))=f(\Phi(\eta(x),x))=f(x)$ for all $x \in E_U$, so we apply the chain rule to obtain that $Df_x = Df_{h(x)} \circ Dh_x$ for all $x \in E|_U$, or equivalently, $$  \left<\nabla f(x),v\right>=\left<\nabla f(h(x)),Dh_x(v)\right> \ \ \forall x \in E|_U,\ v \in T_xE.$$ 
As discussed in the proof of Proposition \ref{PropEquivFiberProps}, it then holds that  $h(E|_U) \subset E_{b_0}$ and thus $ Dh_x(v) \in T_x (E_{b_0}) = \Ver_x(p)$ for all $x \in E|_U$ and $v \in T_xE$. We derive that
\begin{equation}
\label{EqProofEstVert}
\left<\nabla f(x),v\right>_x = \left<\nabla^\ver f(h(x)),Dh_x(v)\right>_{h(x)} \stackrel{\eqref{EqverGradPhig}}{=} \left<(D\Phi_{\eta(x)})_x(\nabla^\ver f (x)),Dh_x(v)\right>_{h(x)}
\end{equation}
for all $x \in U$. Using the chain rule, we compute that 
\begin{align*}
	Dh_x(v) &= D\Phi_{(\eta(x),x)}(D\eta_x(v),v)= (D\Phi^x)_{\eta(x)}(D\eta_x(v))+(D\Phi_{\eta(x)})_x(v).
\end{align*}
for all $x \in E|_U$ and $v \in T_xE$. Inserting this into  \eqref{EqProofEstVert} yields 
\begin{align*}
\left<\nabla f(x),v\right>_x 
	&= \left<(D\Phi_{\eta(x)})_x(\nabla^\ver f (x)) ,(D\Phi^x)_{\eta(x)}(D\eta_x(v))+(D\Phi_{\eta(x)})_x(v)\right>_{h(x)}\\
	&= \left<\nabla^\ver f(x),(D\Phi_{\eta(x)})_x^{-1}(D\Phi^x)_{\eta(x)}D\eta_x(v)+v\right>_x , 
\end{align*}
where we have used the $G$-invariance of the metric. For $v=\nabla f(x)$ we obtain
\begin{align*}
\|\nabla f(x) \|^2  
&= \left<\nabla^\ver f(x),(D\Phi_{\eta(x)})_x^{-1}(D\Phi^x)_{\eta(x)}D\eta_x(\nabla f(x))\right>_x + \|\nabla^\ver f(x)\|^2 \\
&\leq \|\nabla^\ver f(x)\|\cdot \|(D\Phi^x)_{\eta(x)}\|_\op\cdot  \|D\eta_x\|_\op \cdot \|\nabla f(x))\|+\|\nabla^\ver f(x)\|^2.
\end{align*}
where have used the Cauchy-Schwarz inequality and the $G$-invariance of the metric. Since $\Phi$ is $f$-controlled with respect to $C$, we obtain 
\begin{align*}
\|\nabla f(x)\|^2
&\leq C(f(x))\cdot 	\|\nabla^\ver f(x)\|\cdot \|D\eta_x\|_{\op}\cdot \|\nabla f(x)\|+\|\nabla^\ver f(x)\|^2.
\end{align*}
If $\nabla f(x) \neq 0$, then dividing this inequality by $\|\nabla f(x)\|$ and using that $\|\nabla^\ver f(x)\|\leq \|\nabla f(x)\|$ yields 
$$\|\nabla f(x)\| \leq (C(\lambda)\cdot \|D\eta_x\|_{\op}+1) \|\nabla^\ver f(x)\| \quad \forall x \in f^\lambda \cap E|_U.$$
Here, we use that the inequality is imminent if $\nabla f(x)=0$. It remains to consider  $\|D\eta_x\|_\op$ for $x \in E|_U$. By definition, $\eta$ is given as $\eta = \mathrm{inv} \circ s \circ p$, where $\mathrm{inv}\colon G \to G$ denotes the inversion. Thus,
$$ \|D\eta_x\|_{\op} \leq \|D\mathrm{inv}_{s(p(x))}\|_\op\cdot \|Ds_{p(x)}\|_\op \cdot \|Dp_x\|_\op  \qquad \forall x \in E|_U.$$
Since $p$ is a Riemannian submersion, $\|Dp_x\|_\op = 1$ for all $x \in E$. Let $\ell_g,r_g:G \to G$ denote left- and right-multiplication, respectively, by $g \in G$. A simple computation shows  that $$D\mathrm{inv}_g = - (Dr_{g^{-1}})_e \circ (D\ell_{g^{-1}})_g\qquad \forall g \in G.$$ Since the chosen metric on $G$ is bi-invariant, this yields that $\|D\mathrm{inv}_g\|_\op =1$ for all $g \in G$. We derive that $\|D\eta_x\|_\op \leq \|Ds_{p(x)}\|_\op$ for all $x \in f^\lambda \cap E|_U$.  Inserting this into the above inequality shows the claim.
\end{proof}

\begin{prop}
\label{Propfcontvert}
Assume that $B$ is a closed manifold and that $\varphi$ is a transitive $G$-action. Let $f\in C^1(E)$ be $G$-invariant and bounded from below. If $\Phi$ is $f$-controlled, then $\nabla f|_{f^{<\lambda}}$ is vertically proportional for each $\lambda \in \RR$. 
\end{prop}
\begin{proof}
Let $b_0 \in B$ and consider again $\pi_\varphi\colon G \to B$, $\pi_\varphi(g)=\varphi(g,b_0)$. Then $\pi_\varphi$ is a principal $G$-bundle, so every $b \in B$ has an open neighborhood  $V_b$ which admits a smooth local section $\sigma_b\colon V_b \to G$ of $\pi_\varphi$.  Since $B$ is a regular topological space, for each $b \in B$ there is an open neighborhood $U_b \subset B$ of $b$ with $\overline{U_b} \subset V_b$. By assumption $B$ is compact, so there exist finitely many $b_1,\dots,b_m \in B$, for which $\{U_{b_1},\dots,U_{b_m}\}$ is an open cover of $B$. Put $U_i:= U_{b_i}$ and $s_i\colon U_i \to G$, $s_i := \sigma_{b_i}|_{U_i}$ for each $i \in \{1,2,\dots,m\}$. 
	Since each $\overline{U_i}$ is compact and $s_i$ has a smooth extension to $\overline{U_i}$, it follows that
	$$M_i:= \sup_{b \in U_i}\|(Ds_i)_b\|_{\op} < +\infty.  $$
		Let $C:[\inf f,+\infty) \to (0,+\infty)$ satisfy \eqref{Eqfcontr}. Given $\lambda \in \RR$, it then follows from Lemma \ref{LemmaEstvertprop} that 
		$$\|\nabla f(x)\| \leq (C(\lambda) \cdot M_i+1) \cdot \|\nabla^\ver f(x)\| \quad \forall x \in f^\lambda \cap E|_{U_i}, \ i \in \{1,2,\dots,m\}.$$
		Thus, with $C_0 := C(\lambda) \cdot \max \{M_1,\dots,M_m\}+1$, it follows that $\|\nabla f(x)\| \leq C_0 \|\nabla^\ver f(x)\|$ for all $x \in f^\lambda$, which shows that $\nabla f|_{f^{<\lambda}}$ is vertically proportional. 
\end{proof}

This leads to the following Lusternik-Schnirelmann type result. 

\begin{theorem}	
\label{TheoremParamTCLScontr}
	 Let $r \in \NN$ with $r \geq 2$. Assume that $B$ is a closed manifold and that $\varphi$ is a transitive $G$-action on $B$.  Let $f\in C^1(E)$ be $G$-invariant and bounded from below. If $(E,\left<\cdot,\cdot\right>)$ is complete, if $(f, \left<\cdot,\cdot\right>)$ satisfies the PS condition and if $\Phi$ is $f$-controlled, then
\begin{align*}
&\TC_r[p|_{f^\lambda}\colon f^\lambda \to B] \\
&\leq \sum_{\mu \in (-\infty,r\lambda]} \max \left\{ \TC_{r,p_\lambda} (\Crit_{\mu_1} f  \times_B \dots \times_B \Crit_{\mu_r} f) \mid \mu_1,\dots,\mu_r \in (-\infty, \lambda], \, \mu_1+\dots+\mu_r=\mu\right\}
 	\end{align*}
\end{theorem}
\begin{proof}
	This follows from Proposition \ref{Propfcontvert} and Theorem \ref{TheoremParamTCLSprod}.
\end{proof}

If both $E$ and $B$ are closed manifolds, the assumptions of Theorem \ref{TheoremParamTCLScontr} can be simplified drastically.

\begin{cor}
\label{CorEquivPTCLS}
	 Let $r\in \NN$ with $r\geq 2$. Assume that $E$ and $B$ are closed manifolds and that $\varphi$ is a transitive $G$-action on $B$. Then for any $G$-invariant $f \in C^1(E)$ it holds that
$$\TC_r[p\colon E\to B] \leq \sum_{\mu \in \RR} \max \Big\{ \TC_{r,p} (\Crit_{\mu_1} f  \times_B \dots \times_B \Crit_{\mu_r} f) \ \Big| \ \mu_1,\dots,\mu_r \in \RR, \, \sum_{i=1}^n \mu_i =\mu\Big\}.$$
\end{cor}
\begin{proof}
Since $E$ is a closed manifold, $(E,\left<\cdot,\cdot\right>)$ is complete and $(f,\left<\cdot,\cdot\right>)$ satisfies the PS condition. Moreover, $\Phi$ is $f$-controlled, see Remark \ref{Remarkfcontr}. Since $E$ is compact, $f$ attains its maximum $\lambda$, so that $p = p|_{f^\lambda}$. Then claim is thus an immediate consequence of Theorem \ref{TheoremParamTCLScontr}.
\end{proof}

In the rest of this section, we want to apply Corollary \ref{CorEquivPTCLS} to unit tangent bundles of odd-dimensional spheres. The parametrized TC of unit tangent bundles of odd-dimensional spheres has been computed by Minowa in \cite{Minowa}. A general upper bound of the parametrized TCs of unit sphere bundles of vector bundles has been obtained by Farber and Weinberger in \cite{farber2023parametrized} and  extended to sequential TCs by Farber and Paul in \cite{FarberPaulSphere}. Using our approach, we will compute the values of the sequential parametrized TCs of unit tangent bundles of $(4m-1)$-dimensional spheres, where $m \in \NN$.\medskip 

In the following we fix $n \in \NN$ and will use the canonical identification $\RR^{2n}\cong \CC^n$ without always spelling it out explicitly. We further view $\S^{2n-1}$ as equipped with the metric induced by the standard metric on $\RR^{2n}$ and denote its unit tangent bundle by 
 $p\colon U\S^{2n-1}\to \S^{2n-1}.$ 
 Explicitly, 
 $$U\S^{2n-1} = \{(x,v) \in \S^{2n-1} \times \RR^{2n} \mid \langle x,v\rangle=0, \ \langle v,v\rangle =1\},$$
 where $\left<\cdot,\cdot\right>$ denotes the Euclidean inner product on $\RR^{2n}$. Seen as a subgroup of $SO(2n)$, the special unitary group $SU(n)$ acts transitively on the sphere. It further induces an action on $U\S^{2n-1}$, which is not transitive.  
The map $p$ is $SU(n)$-equivariant and scalar multiplication with the complex unit induces an $SU(n)$-equivariant section 
$$ A\colon \S^{2n-1} \to U\S^{2n-1}, \qquad      A(x) =  (x,i x) ,       $$
of $p$. 
We note that the scalar multiplication with $i$ is explicitly given by
\begin{equation}\label{eq_multiplication_with_i}
        ix=   i(x_1,\ldots, x_{2n}) = (x_2,-x_1,x_4,-x_3, \ldots, x_{2n},-x_{2n-1})      
\end{equation}
for $x= (x_1,\ldots, x_{2n})\in \RR^{2n}$. We can identify $U\S^{2n-1}$ with the Stiefel manifold of orthonormal $2$-frames in $\RR^{2n}$, 
$$   V_2(\RR^{2n}) =   \{  X \in M_{2n,2}(\RR) \mid  X^T X =  \mathbb{1}_2 \}      $$
where $M_{k,m}(\RR)$ denotes the space of all real $k\times m$-matrices and where $\mathbb{1}_2 \in M_{2,2}(\RR)$ denotes the identity matrix.
Explicitly, we identify a unit tangent vector $(x,v)\in U\S^{2n-1}$ with the matrix $X = (x,v)$ where $x\in\RR^{2n}$ and $v\in \RR^{2n}$ are taken as the columns of $X$.
The tangent space at an element $X\in U\S^{2n-1}$ is then identified with
$$     T_X U\S^{2n-1} =  \{ Y\in M_{2n,2}(\RR) \mid  X^T Y + Y^T X =  0 \} .     $$
In particular, if we write $X\in U\S^{2n-1}$ as a matrix $X= (x_1,x_2)$ with columns $x_1,x_2\in \RR^{2n}$ and $Y\in T_X U\S^{2n-1}$ as $Y = (y_1,y_2)$ with columns $y_1,y_2\in \RR^{2n}$ we have
\begin{equation}\label{eq_tangent_space_unit_tangent_bundle}
           \langle x_1,y_1\rangle = \langle x_2,y_2\rangle = 0      \quad \text{as well as }\quad    \langle x_2,y_1 \rangle  + \langle x_1, y_2\rangle = 0 .
\end{equation}
 We consider the function\begin{equation}
\label{EqfUS2n-1}
    f\colon U \S^{2n-1} \to \RR, \quad f(X) =   \langle A(x_1),x_2\rangle= \langle ix_1,x_2\rangle \quad \text{for}\quad X= (x_1,x_2)\in U\S^n .  
\end{equation}
Since $A$ is $SU(n)$-equivariant and since the Euclidean inner product is $SO(2n)$-invariant and thus $SU(n)$-invariant as well, $f$ is $SU(n)$-invariant. 
Both $U\S^{2n-1}$ and $\S^{2n-1}$ are closed manifolds, so we can apply Corollary \ref{CorEquivPTCLS} to this setting and obtain for each $r\geq 2$ that 
\begin{equation}
\label{EqPTCupperUS2n-1}
\TC_r[p] \leq \sum_{\mu \in \RR} \max \Big\{\TC_{r,p}(\Crit_{\mu_1}f\times_{\S^{2n-1}} \dots \times_{\S^{2n-1}} \times \Crit_{\mu_r}f) \ \Big| \ \sum_{j=1}^r \mu_j = \mu\Big\},
\end{equation}
where $\TC_r[p]:=\TC_r[p\colon U\S^{2n-1} \to \S^{2n-1}]$. We want to further compute this upper bound of $TC_r[p]$ and study the critical sets of $f$ in greater detail for this purpose. 
\begin{lemma}\label{lemma_crit_points_unit_sphere_bundle}
Let $f\colon U\S^{2n-1} \to \RR$ be given as in \eqref{EqfUS2n-1}.    Then $\Crit f = \Crit_1 f \cup \Crit_{-1} f$, where
    $$       \Crit_1 f =  \{ (x,i x)\in U\S^{2n-1} \mid x\in \S^{2n-1}\}  \quad \text{and}\quad      \Crit_{-1} f =  \{ (x,-i x)\in U\S^{2n-1} \mid x\in \S^{2n-1}\}   .   $$
\end{lemma}
\begin{proof}
We compute the differential of $f$ at $X = (x_1,x_2) \in U\S^{2n-1}$ as
    \begin{equation}\label{eq_differential_f}
        Df_X (Y) =  \langle i y_1, x_2 \rangle + \langle i x_1, y_2 \rangle     \qquad \forall Y=(y_1,y_2) \in T_XU\S^{2n-1}.
    \end{equation}
    Let $x\in \S^{2n-1}$. Evidently, $f(x,ix) =1$ and for each $Y = (y_1,y_2)\in T_X U\S^{2n-1}$ we have
    $$   Df_{(x,ix)} (Y) =  \langle ix,i y_1 \rangle +  \langle  i x, y_2\rangle = - \langle x,y_1\rangle + \langle ix,y_2\rangle = 0    $$
by equation \eqref{eq_tangent_space_unit_tangent_bundle}.
    Hence, $(x,ix)\in \Crit_1 f$ and similarly one shows that for each $x\in \S^{2n-1}$ it holds that  $(x,-ix)\in \Crit_{-1} f$. Thus, we have shown that
    $$   \{ (x,i x)\in U\S^{2n-1}\,|\, x\in \S^{2n-1}\}  \cup \{ (x,-i x)\in U\S^{2n-1}\,|\, x\in \S^{2n-1}\}    \subset \Crit f .    $$
To show the converse inclusion, let $X=(x_1,x_2)\in U\S^{2n-1}$ with $x_2\notin \{ix_1,-ix_1\}$.
    Set $Y = (0, w)$ with $w = ix_1 - \langle i x_1, x_2\rangle x_2$.
    One verifies that indeed $Y\in T_X U\S^{2n-1}$ and it holds that
    $   Df_X (Y)  =  1 - \langle ix_1, x_2\rangle^2 \neq 0      $.
    Hence, $X\not\in \Crit f$, which shows that any critical point of $f$ is of the form derived above.  
\end{proof} 
Let $r \in \NN$ with $r \geq 2$. We have seen in the previous Lemma that the critical values of $f$ are $1$ and $-1$, so the fiber products of the critical manifolds occurring in \eqref{EqPTCupperUS2n-1} are of the form
$$     \Crit_{\mu_1}f \times_{\S^{2n-1}} \Crit_{\mu_2} f \times_{\S^{2n-1}}\ldots     \times_{\S^{2n-1}} \Crit_{\mu_r} f     \quad \text{with} \quad \mu_j\in\{+1,-1\},\,\,j\in\{1,2,\ldots, r\} . $$
The sums of critical values that occur as values of $\mu$ on the right-hand side of \eqref{EqPTCupperUS2n-1} must further lie in $\{-r,-r+2,\dots,r-2,r\}$. Thus, it follows from \eqref{EqPTCupperUS2n-1} that 
\begin{equation}
\label{EqTCupperUSn}
\TC_r[p] \leq \sum_{k=0}^r \max \Big\{\TC_{r,p}(\Crit_{\mu_1}f \times_{\S^{2n-1}} \dots \times_{\S^{2n-1}} \Crit_{\mu_r}f) \ \Big| \   \sum_{j=1}^r \mu_j = 2k-r \Big\}. 
\end{equation}
We want to explicitly compute the sequential parametrized subspace complexities that occur on the right-hand side of this inequality in the case of even $n$, i.e. for spheres of dimension $4m-1$ for some $m\in \NN$. Eventually, we will show the following.

\begin{theorem}
\label{TheoremPTC4m-1}
For each $m \in \NN$, it holds that
$\TC_r[p\colon U\S^{4m-1} \to \S^{4m-1}] = r+1. $	
\end{theorem}

For this purpose, we will use the canonical identification $\RR^4 \cong \HH$ with $\HH$ denoting the quaternions, and consider $\RR^{4m}\cong \HH^m$ as a left $\HH$-module. 
The quaternionic units $i,j,k\in \mathbb{H}$ can be chosen in such a way that multiplication with $i$ is again given by equation \eqref{eq_multiplication_with_i} and that left-multiplication with $j\in \mathbb{H}$ on $\RR^{4m}$ is given by
$$    j(x_1,\ldots, x_{4m}) =   (-x_4,-x_3,x_2,x_1, -x_8,-x_7,x_6,x_5,\ldots, -x_{4m},-x_{4m-1},x_{4m-2},x_{4m-3})  .      $$
In particular, it holds that 
$       \langle x, jx\rangle =  \langle ix, jx \rangle = 0  $ and $ \|jx\|^2  = \|x\|^2 
       $ for all $x\in \mathbb{R}^{4m}$.
\begin{lemma}
\label{LemmaTCParamCrit}
Let $m \in \NN$ and let $f\colon \S^{4m-1}\to \RR$ be given as in \eqref{EqfUS2n-1}.
    Let $\mu_1,\ldots,\mu_r\in\{+1,-1\}$ and consider the manifold $\mathcal{C} =   \Crit_{\mu_1}f \times_{\S^{4m-1}} \Crit_{\mu_2} f \times_{\S^{4m-1}}\ldots     \times_{\S^{4m-1}} \Crit_{\mu_r} f.  $
    Then $\TC_{r,p}(\C) = 1$.
\end{lemma}
\begin{proof}
Consider the map $B\colon \S^{4m-1}\to U\S^{4m-1}$, $B(x) = jx$.
   The map $B$ satisfies $\langle A(x),B(x)\rangle = 0$ for all $x\in\S^{4m-1}$.
     Let $u  = ((x,v_1),\ldots,(x, v_r)) \in \C $, so that $v_j \in \{ix,-ix\}$ for all $j \in \{1,2,\ldots, r\}$.
    We define a path $ \sigma_u\colon I\to U\S^n      $ as follows.
    For all $j\in \{0,1,\ldots , r-2\}$ and  $t\in [\frac{j}{r-1},\frac{j+1}{r-1}]$ we put 
    $$\sigma_u(t) = \begin{cases}
    	(x,v_{j+1}) & \text{if } v_{j+2}=v_{j+1}, \\
    	\big(x, \cos(\pi((r-1)t -j)) v_{j+1} + \sin( \pi((r-1)t-j)) B(x) \big) & \text{if } v_{j+2} =-v_{j+1}.
   \end{cases}$$
    Clearly, $\sigma_u$ is a continuous path whose image lies in $p^{-1}(\{x\})$. One further checks that the fibration $\Pi_r\colon (U\S^{4m-1})^I_{\S^{4m-1}} \to (U\S^{4m-1})^r_{\S^{4m-1}}$ satisfies
    $$   \Pi_r ( \sigma_u)  =  ((x,v_1),(x,v_2),\ldots, (x,v_r)) .    $$
    Moreover, the map $s\colon \C  \to (U\S^{4m-1})_{\S^{4m-1}}^I$, $s(u) = \sigma_u$, is continuous and it follows that $s$ is a continuous sequential parametrized motion planner on $\C$. Since $\C$ is closed in $(U\S^{4m-1})^r_{\S^{4m-1}}$ and since $U\S^{4m-1}$ and $\S^{4m-1}$ are closed connected manifolds, hence path-connected ANRs, it follows from Lemma \ref{LemmaPMPextend} that $s$ extends to a continuous sequential parametrized motion planner on an open neighborhood of $\C$. Hence, $\TC_{r,p}(\C)=1$.
\end{proof}

\begin{proof}[Proof of Theorem \ref{TheoremPTC4m-1}]
It follows from Lemma \ref{LemmaTCParamCrit} that each summand of the form 
$$  \max \Big\{ \TC_{r,p} (\Crit_{\mu_1} f  \times_{\S^{4m-1}} \dots \times_{\S^{4m-1}} \Crit_{\mu_r} f) \ \Big| \  \sum_{j=1}^r \mu_j =2k-r\Big\}      $$
occurring on the right-hand side of \eqref{EqTCupperUSn} equals one. Since there are precisely $r+1$ of those summands, this yields 
$$    \TC_r[p\colon U\S^{4m-1}\to \S^{4m-1} ] \leq r +1   .  $$	
It is shown in \cite[p. 9]{FarberPaulI} that the $r$-th sequential parametrized TC of a fibration is bounded from below by the $r$-th sequential TC of its fiber. Each fiber of $p\colon U\S^{4m-1} \to \S^{4m-1}$ is homeomorphic to $\S^{4m-2}$, so we derive that 
$$\TC_r[p\colon U\S^{4m-1} \to \S^{4m-1}] \geq \TC_r(\S^{4m-2})=r+1,$$
using the computation of sequential TCs of spheres from \cite[Section 4]{RudyakHigher}. Combining the lower and the upper bound shows the claim.
\end{proof}

\begin{remark}
The fibration $p\colon U\S^{4m-1}\to \S^{4m-1}$ is fiber homotopy equivalent to a trivial fibration if and only if $m \in \{1,2\}$, so that Theorem \ref{TheoremPTC4m-1} can not be derived by elementary arguments for $m \geq 3$. More precisely, as shown in \cite[Theorem 1.1]{james1962note}, $U\S^{4m-1}$ has the homotopy type of $\S^{4m-1}\times \S^{4m-2}$ if and only if there exists an element of Hopf invariant one in $\pi_{8m-1}(\S^{4m})$, which  By \cite{adams1960non} this is the case only for $m\in \{1,2\}$.
    Thus, $p$ is not fiber homotopy trivial for $m \geq 3$. 
\end{remark}

\bibliographystyle{amsalpha}

\bibliography{TCloop.bib}

\end{document}